\documentclass[reqno]{amsart}

\usepackage[margin=1in]{geometry}

\usepackage{amscd,amssymb, amsmath, wasysym}
\usepackage{graphicx}
\usepackage{amsfonts}
\usepackage{mathrsfs}    
\usepackage{amsmath}    
\usepackage{amsthm}     
\usepackage{amscd}      
\usepackage{amssymb}    
\usepackage{latexsym}   
\usepackage{graphicx}   
\usepackage{verbatim}   
\usepackage[all]{xy}     
\usepackage{xfrac}
\usepackage{amsmath,amssymb,fancyhdr,color,epsfig}

\usepackage[T2A]{fontenc}
\usepackage[utf8]{inputenc}
\usepackage[english]{babel}

\usepackage{textcase}
\usepackage{wrapfig}

\newtheorem{theorem}[subsection]{Theorem}
\newtheorem{lemma}[subsection]{Lemma}

\newtheorem{corollary}[subsection]{Corollary}

\newtheorem{remark}[subsection]{Remark}
\newtheorem*{remark*}{Remark}

\newtheorem{definition}[subsection]{Definition}

\makeatletter
\@addtoreset{equation}{section}
\@addtoreset{figure}{section}
\@addtoreset{table}{section}
\makeatother

\makeatletter
\newcommand\testshape{family=\f@family; series=\f@series; shape=\f@shape.}
\def\myemphInternal#1{\if n\f@shape%
\begingroup\itshape #1\endgroup\/%
\else\begingroup\bfseries #1\endgroup%
\fi}
\def\myemph{\futurelet\testchar\MaybeOptArgmyemph}
\def\MaybeOptArgmyemph{\ifx[\testchar \let\next\OptArgmyemph
                 \else \let\next\NoOptArgmyemph \fi \next}
\def\OptArgmyemph[#1]#2{\index{#1}\myemphInternal{#2}}
\def\NoOptArgmyemph#1{\myemphInternal{#1}}
\makeatother

\newcommand\Mman{M}
\newcommand\Nman{N}
\newcommand\Uman{U}

\newcommand\Xman{X}

\newcommand\RRR{\mathbb{R}}
\newcommand\ZZZ{\mathbb{Z}}
\newcommand\FFF{\mathcal{F}}

\newcommand\eps{\varepsilon}

\newcommand\id{\mathrm{id}}          
\newcommand\image{\mathrm{image}}    
\newcommand\Int{\mathrm{Int}}        
\newcommand\supp{\mathrm{supp\,}}    
\newcommand\wrm[1]{\mathop{\wr}\limits_{\ZZZ_{#1}}}
\newcommand\Maps{\mathrm{Map}}

\newcommand\Diff{\mathcal{D}}
\newcommand\Stab{\mathcal{S}}
\newcommand\Orbit{\mathcal{O}}
\newcommand\DiffId{\mathcal{D}_{\id}}
\newcommand\StabId{\mathcal{S}_{\id}}

\newcommand\DiffM{\Diff(M)}
\newcommand\Stabf{\Stab(f)}
\newcommand\Orbf{\Orbit(f)}
\newcommand\Orbff{\Orbit_{f}(f)}

\newcommand\StabIdf{\StabId(f)}

\newcommand\func{f}
\newcommand{\fKRGraph}{\Gamma(\func)}

\newcommand\flow{\mathbf{F}}

\newcommand\Torus{T^2}

\newcommand\curveMeridian{C}
\newcommand\curveParallel{C'}
\newcommand\flowMeridian{\mathbf{M}}
\newcommand\flowParallel{\mathbf{L}}

\newcommand\subgroupMeridian{\mathcal{M}}
\newcommand\subgroupParallel{\mathcal{L}}

\newcommand\DiffT{\Diff(\Torus)}
\newcommand\DiffIdT{\DiffId(\Torus)}
\newcommand\DiffTC{\Diff(\Torus,\CrvList)}
\newcommand\DiffIdTC{\DiffId(\Torus,\CrvList)}

\newcommand\StabPrf{\Stab'(\func)}

\newcommand\StabfX{\Stab(\func,\Xman)}
\newcommand\StabIdfX{\StabId(\func,\Xman)}
\newcommand\StabPrfX{\Stab'(\func,\Xman)}
\newcommand\DiffMX{\Diff(M,\Xman)}
\newcommand\DiffIdMX{\DiffId(M,\Xman)}
\newcommand\OrbfX{\Orbit(\func,\Xman)}
\newcommand\OrbffX{\Orbit_{\func}(\func,\Xman)}

\newcommand\SerreFibr{p}
\newcommand\Morse{\mathrm{Morse}}

\newcommand\eval{\varphi}

\newcommand\kerjo{\mathcal{K}}

\newcommand\incmap{i}
\newcommand\jZ{\incmap_0}
\newcommand\jO{\incmap_1}

\newcommand\fibr{\lambda}

\newcommand\qhom{q} 

\newcommand\DT{\Diff^{\id}}

\newcommand\DTC{\Diff^{\id}_{\CrvList}}

\newcommand\Of{\Orbit}

\newcommand\OfC{\Orbit_{\CrvList}}

\newcommand\Sf{\Stab}

\newcommand\SfC{\Stab_{\CrvList}}
\newcommand\Sidf{\Stab^{\id}}
\newcommand\SidfC{\Stab^{\id}_{\CrvList}}

\newcommand\Cyl{Q}
\newcommand\CrvList{\mathcal{C}}
\newcommand\dif{h}
\newcommand\gdif{g}

\newcommand\slide{\theta}
\newcommand\dtw{\tau}
\newcommand\grp{\mathcal{G}}

\newcommand\qiso{q}
\newcommand\CNbh{W}

\newcommand\Slide{\Theta}
\newcommand\krot{\kappa}
\newcommand\dc{\partial_{\CrvList}}
\newcommand{\idT}{\id_{\Torus}}
\newcommand\pmult{\,\,}

\newcommand\idQ{\id_{\Cyl_0}}
\newcommand\dQ{\partial\Cyl_0}
\newcommand\SfQ{\Stab^{\Cyl}}
\newcommand\OfQ{\Orbit^{\Cyl}}
\newcommand\DQ{\Diff^{\Cyl}}
\newcommand\PQ{\mathcal{P}}

\newcommand\cclass[1]{[#1]_{c}}
\newcommand\ncover{\eta}
\newcommand\hfunc{\widehat{\func}}

\begin{document}

\title[Smooth functions on $2$-torus]
{Smooth functions on $2$-torus whose Kronrod-Reeb graph contains a cycle}

\author{Sergiy Maksymenko}
\address{Topology department, Institute of Mathematics of NAS of Ukraine, Tereshchenkivska str. 3, Kyiv, 01601, Ukraine}
\curraddr{}
\email{maks@imath.kiev.ua}

\author{Bohdan Feshchenko}
\address{Topology department, Institute of Mathematics of NAS of Ukraine, Tereshchenkivska str. 3, Kyiv, 01601, Ukraine}
\email{fb@imath.kiev.ua}


\begin{abstract}
Let $f:M\to \RRR$ a Morse function on a connected compact surface $M$, and $\mathcal{S}(f)$ and $\mathcal{O}(f)$ be respectively the stabilizer and the orbit of $f$ with respect to the right action of the group of diffeomorphisms $\mathcal{D}(M)$.
In a series of papers the first author described the homotopy types of connected components of $\mathcal{S}(f)$ and $\mathcal{O}(f)$ for the cases when $M$ is either a $2$-disk or a cylinder or $\chi(M)<0$.
Moreover, in two recent papers the authors considered special classes of smooth functions on $2$-torus $T^2$ and shown that the computations of $\pi_1\mathcal{O}(f)$ for those functions reduces to the cases of $2$-disk and cylinder.

In the present paper we consider another class of Morse functions $f:T^2\to\RRR$ whose KR-graphs have exactly one cycle and prove that for every such function there exists a subsurface $Q\subset T^2$, diffeomorphic with a cylinder, such that $\pi_1\mathcal{O}(f)$ is expressed via the fundamental group  $\pi_1\mathcal{O}(f|_{Q})$ of the restriction of $f$ to $Q$.

This result holds  for a larger class of smooth functions $f:T^2\to \RRR$ having the following property: for every critical point $z$ of $f$ the germ of $f$ at $z$ is smoothly equivalent to a homogeneous polynomial $\RRR^2\to \RRR$ without multiple factors.
\end{abstract}



\maketitle

\begin{center}
\em Dedicated in memory of our teacher Sharko Volodymyr Vasylyovych
\end{center}

\medskip

\section{Introduction}
Let $\Mman$ be a smooth compact surface, $\Xman \subset\Mman$ be a closed (possibly empty) subset, and $\DiffMX$ be the group of diffeomorphisms of $\Mman$ fixed on some neighbourhood of $\Xman$.
Then $\DiffMX$ acts from the right on $C^{\infty}(\Mman)$ by following rule:
if $\dif\in\DiffMX$ and $\func\in C^{\infty}(\Mman)$ then the result of the action of $\dif$ on $\func$ is the composition map
\begin{equation}\label{main-act}
\func\circ\dif : M\xrightarrow{~~\dif~~} M \xrightarrow{~~\func~~}\RRR.
\end{equation}
Given $f\in C^{\infty}(\Mman)$ let
\begin{align*}
\StabfX &= \{\func\in\DiffMX\mid\func\circ\dif=\func\}, &
\OrbfX &= \{\func\circ\dif\mid\dif\in\DiffMX\}.
\end{align*}
be respectively the \emph{stabilizer} and the \emph{orbit} of $\func$ under the action~\eqref{main-act}.
Let also
\[ \StabPrfX = \Stabf \cap \DiffIdMX.\]
If $\Xman$ is empty, then we omit it from notation and write $\DiffM=\Diff(M,\varnothing)$, $\Stabf=\Stab(f,\varnothing)$, $\Orbf=\Orbit(f,\varnothing)$, and so on.
We will also endow the spaces $\DiffMX$, $C^{\infty}(\Mman)$, $\StabfX$, and $\OrbfX$ with the corresponding Whitney $C^{\infty}$-topologies.

Denote by $\StabIdfX$ and $\DiffIdMX$ the identity path components $\StabfX$ and $\DiffMX$ respectively, and $\OrbffX$ be the path component of $\func$ in $\OrbfX$.

\smallskip 

Let $\mathcal{F}(M)$ be a subset in $C^{\infty}(\Mman)$ consisting of functions $\func$ having the following two properties:
\begin{itemize}
\item[\rm(B)]
$\func$ takes a constant value at each connected components of $\partial\Mman$, and all critical points of $\func$ are contained in the interior of $M$;

\item[\rm(L)]
for every critical point $z$ of $\func$ the germ of $\func$ at $z$ is smoothly equivalent to a certain {\bfseries homogeneous polynomial $f_z:\RRR^2\to\RRR$ without multiple factors}.
\end{itemize}

Let $\Morse(M) \subset C^{\infty}(\Mman)$ be an open and everywhere dense subset consisting of all Morse functions having the above property (B), that is functions having only \myemph{non-degenerate} critical points.
By the Morse lemma every non-degenerate singularity is smoothly equivalent to a homogeneous polynomial $\pm x^2 \pm y^2$ having no multiple factors.
Therefore $\Morse(M) \ \subset \ \FFF(\Mman)$.
This shows that the class $\FFF(\Mman)$ is large.

\smallskip

Let $\func\in \FFF(\Mman)$ and $c\in\RRR$.
A connected component $\curveMeridian$ of the level set $\func^{-1}(c)$ is called \myemph{critical} if $\curveMeridian$ contains at least one critical point of $\func$; otherwise $\curveMeridian$ is \myemph{regular}.
Consider a partition $\Delta$ of $\Mman$ into connected component of level sets of $\func$.
It is well known that the corresponding quotient $\Mman/\Delta$ has a structure of a finite one-dimensional $CW$-complex and is called \emph{Kronrod-Reeb graph} or simply KR-graph of the function $\func$.
We will denote it by $\fKRGraph$.
The vertices of $\fKRGraph$ are critical components of level sets of $\func$.

This graph was introduced by A.~S.~Kronrod in~\cite{Kronrod:UMN:1950} for studying functions on surfaces and also used by by G.~Reeb in~\cite{Reeb:ASI:1952}.
Applications of $\fKRGraph$ to study Morse functions on surfaces are given e.g. in~\cite{BolsinovFomenko:1997, Kulinich:MFAT:1998, Kudryavtseva:MatSb:1999, Sharko:UMZ:2003, Sharko:MFAT:2006, MasumotoSaeki:KJM:2011}.

\medskip

In a series of papers, \cite{Maksymenko:AGAG:2006}, \cite{Maksymenko:ProcIM:ENG:2010}, \cite{Maksymenko:UMZ:ENG:2012}, \cite{Maksymenko:DefFuncI:2014}, \cite{Maksymenko:pi1Repr:2014}, \cite{Maksymenko:orbfin:2014}, the first author calculated the homotopy types of spaces $\Stabf$ and $\Orbf$ for all $\func\in\FFF(\Mman)$, see~\S\ref{sect:hom_types} for some details.
In particular, it was proved, \cite[Theorem~1.5(3)]{Maksymenko:AGAG:2006}, that if $\func$ is a \myemph{generic} Morse function, i.e. it takes distinct values at distinct critical point, then $\Orbff$ is homotopy equivalent to a finite-dimensional torus.

This result was improved by E.~Kudryavtseva \cite[Theorem~2.5(B)]{Kudryavtseva:MathNotes:2012}, \cite[Theorem~2.6(C)]{Kudryavtseva:MatSb:2013}: using another approach she shown that if $\Mman$ is orientable, $\chi(\Mman)<0$, and $\func$ is Morse, then $\Orbff$ is homotopy equivalent to a quotient $T^{k}/G$ of a finite-dimensional torus $T^k$ by the free action of some finite group $G$.

Recently, \cite{Maksymenko:orbfin:2014}, the first author established such a statement for all $\func\in\FFF(\Mman)$ provided $\Mman$ is distinct from $2$-torus, $2$-sphere, projective plane, and Klein bottle.
It was also shown in~\cite[Theorem~1.8]{Maksymenko:MFAT:2010} that under the same restrictions on $\Mman$, the computation of the homotopy type of $\Orbf$, reduces to the case when $\Mman$ is either $2$-disk, or a cylinder, or a M\"obius band.

In two recent papers, \cite{MaksymenkoFeshchenko:UMZ:ENG:2014}, \cite{MaksymenkoFeshchenko:MS:2014}, the authors studied smooth functions on $2$-torus and shown that under some conditions on $\func\in\FFF(\Torus)$ the computation of the homotopy type of $\Orbf$ also reduces to the cases when $\Mman$ is a $2$-disk or a cylinder.

In the present paper we study functions $\func\in\FFF(\Torus)$ whose Kronrod-Reeb graph has one cycle.
The main result, see Theorem~\ref{th:main:pi1Of}, reduces the computation of $\Orbff$ to the restriction of $\func$ onto some subsurface $\Cyl\subset\Torus$ diffeomorphic to a cylinder.
We also give exact formula expressing $\pi_1\Orbff$ via $\pi_1\Orbit(\func|_{\Cyl})$.
This extends the result of~\cite{MaksymenkoFeshchenko:MS:2014}.

\begin{remark}\rm
In~\cite{MaksymenkoFeshchenko:MS:2014} the group $\DiffMX$ means the group of diffeomorphisms \myemph{fixed on $\Xman$}, while in the present paper we denote by $\DiffMX$ the group of diffeomorphisms \myemph{fixed on some neighbourhood} of $\Xman$.
In fact, if $\Xman$ is a finite collection of regular components of some level-sets of $\func\in\FFF(\Mman)$, such a restriction does not impact on the homotopy types of $\DiffMX$, $\StabfX$ and $\Orbf$, see~\cite{Maksymenko:UMZ:ENG:2012}.
\end{remark}

\subsection{Wreath products $G\wrm{n}\ZZZ$}\label{sect:wreath_product}
Let $G$ be a group with unit $e$, and $n\geq1$.
Denote by $\Maps(\ZZZ_n,G)$ the group of all \myemph{maps}, not necessarily homomorphisms, from cyclic group $\ZZZ_n$ into $G$, with respect to point wise multiplication.
That is if $\alpha,\beta:\ZZZ_n\to G$ two elements from $\Maps(\ZZZ_n,G)$, then their product is given by the formula $(\alpha\beta)(i) = \alpha(i)\cdot\beta(i)$ for $i\in\ZZZ_n$, where the multiplication $\cdot$ is taken in the group $G$.

Notice that the group $\ZZZ$ acts from the right on $\Maps(\ZZZ_n,G)$ by the following rule: if $\alpha\in\Maps(\ZZZ_n,G)$ and $a\in\ZZZ$, then the result $\alpha^{k}:\ZZZ_n\to G$ of the action of $k$ on $\alpha$ is given by the formula:
\begin{equation}\label{equ:action_k_on_alpha}
\alpha^{k}(i) = \alpha(i + k \ \mathrm{mod} \ n), \qquad i\in\ZZZ_n.
\end{equation}
The semidirect product $\Maps(\ZZZ_n,G) \rtimes \ZZZ$ corresponding to this action is called a \myemph{wreath product of $G$ and $\ZZZ$ over $\ZZZ_n$} and denoted by 
\[ G\wrm{n}\ZZZ:= \Maps(\ZZZ_n,G) \rtimes \ZZZ.\]
More precisely, $G\wrm{n}\ZZZ$ is the \myemph{set} $\Maps(\ZZZ_n,G) \times \ZZZ$ with the following operation
\begin{equation}\label{equ:mult_in_wreath_prod}
(\alpha,k)\,(\beta,l) = (\alpha\beta^{k}, k+l),
\end{equation}
for all $(\alpha,k), (\beta,l)\in\Maps(\ZZZ_n,G) \times \ZZZ$.

In particular, we have the following short exact sequence:
\begin{equation}\label{equ:exact_seq_for_wreath_prod}
1 \longrightarrow \Maps(\ZZZ_n, G) \xrightarrow{~~\zeta~~} G \wrm{n} \ZZZ \xrightarrow{~~p~~} \ZZZ \longrightarrow 1,
\end{equation}
where $\zeta(\alpha)=(\alpha,0)$ is a \myemph{canonical inclusion} and $p(\alpha,k)=k$ is a \myemph{canonical projection}.

Notice also that for $n=1$, there is a natural isomorphism $G\wrm{n}\ZZZ \ \cong G \times \ZZZ$.

\subsection{Parallel curves on $\Torus$}\label{sect:paral_curves}
A finite non-empty family of $\curveMeridian_0,\ldots,\curveMeridian_{n-1} \subset \Torus$ of simple closed curves will be called \myemph{parallel} if these curves are mutually disjoint and non-separating.

If $n=1$, then $\Torus\setminus\curveMeridian$ is an open cylinder, we will regard $\Torus$ as a cylinder $\Cyl_0$ with identified boundary components, see Figure~\ref{fig:paral_curves}a).

Suppose $n\geq2$. 
Then all curves in a parallel family must be isotopic each other.
In this case we will always assume that they are \myemph{cyclically enumerated along $\Torus$}, that is $\curveMeridian_i$ and $\curveMeridian_{i+1}$ bound a cylinder $\Cyl_{i}$ containing no other curves $\curveMeridian_j$, where all indices are taken modulo $n$, see Figure~\ref{fig:paral_curves}b).
We will also use the following notation:
\begin{align*}
\CrvList &= \mathop{\cup}\limits_{i=0}^{n-1}\curveMeridian_i, &
\curveMeridian_i &:= \curveMeridian_{i\,\mathrm{mod}\,n}, &
\Cyl_i &:= \Cyl_{i\,\mathrm{mod}\,n}
\end{align*}
for all integers $i\in\ZZZ$.
\begin{figure}[h]
\includegraphics[height=2.5cm]{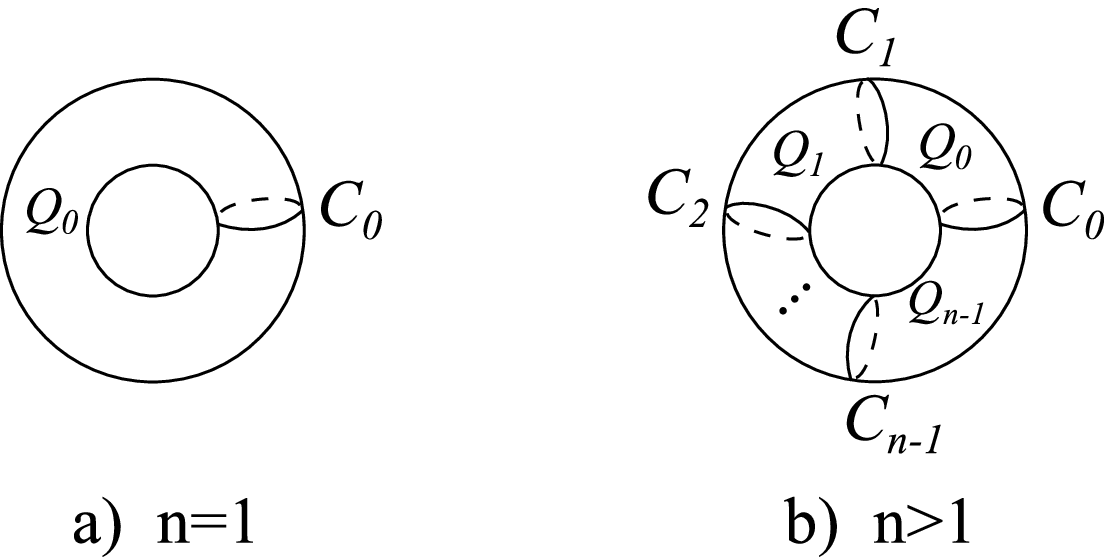}
\caption{}\label{fig:paral_curves}
\end{figure}

\subsection{Cyclic index of $\func$}\label{sect:cyclic_index}
Let $\func\in\FFF(\Torus)$ be such that its KR-graph $\fKRGraph$ is not a tree.
It is easy to show, \cite{MaksymenkoFeshchenko:MS:2014}, that then $\fKRGraph$ has a unique simple cycle, which we will denote by $\Lambda$, see Figure~\ref{fig:f_curves}.
\begin{figure}[h]
\includegraphics[height=2.5cm]{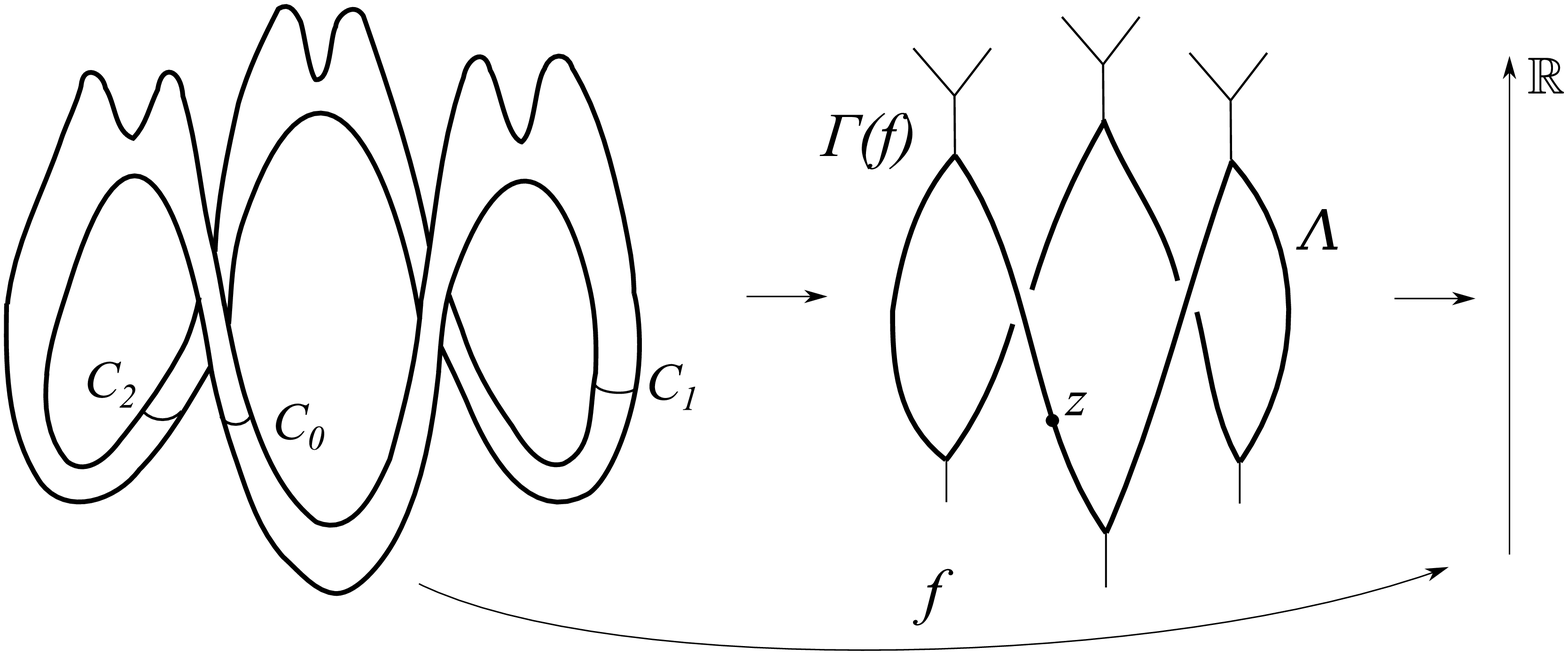}
\caption{}\label{fig:f_curves}
\end{figure}
Let also $\curveMeridian \subset\Torus$ be a regular component of some level set $\func^{-1}(c)$, $c\in\RRR$, and $z$ be the corresponding point on $\fKRGraph$.
It is easy to check, see~\cite{MaksymenkoFeshchenko:MS:2014}, that $z\in\Lambda$ if and only if $\curveMeridian$ does not separate $\Torus$.
Notice that $\func^{-1}(c)$ consists of finitely many connected components and is invariant with respect to each $\dif\in\Stabf$.
Let 
\[
\CrvList = \{\dif(\curveMeridian) \mid \dif\in\StabPrf \}.
\]
be the set of images of $\curveMeridian$ under the action of $\StabPrf = \Stabf \cap \DiffIdT$.
Then $\CrvList$ consists of finitely many connected components of $\func^{-1}(c)$:
\[
\CrvList  \ =  \ \{\, \curveMeridian_0=\curveMeridian, \ \curveMeridian_1, \ \ldots,  \ \curveMeridian_{n-1}\, \}
\]
for some $n\geq1$.
Emphasize that we only consider the images of $\curveMeridian$ for all diffeomorphisms $\dif$ that preserve $\func$ and are \myemph{isotopic to $\CrvList$}.
However, there may exist $\dif\in\Stabf$ that is not isotopic to $\idT$ and such that $\dif(\curveMeridian) \subset \func^{-1}(c) \setminus \CrvList$.

It follows that the curves in $\CrvList$ are mutually disjoint, and neither of them separates $\Torus$, since $\curveMeridian$ does not do this.
Thus they are \myemph{parallel} in the sense of \S\ref{sect:paral_curves}, and therefore we will assume that they are cyclically ordered along $\Torus$, and that $\curveMeridian_i$ and $\curveMeridian_{i+1}$ bound a cylinder $\Cyl_i$ whose interior does not intersect $\CrvList$.

\begin{definition}\label{def:triv_act_S_on_C}
The number $n$ of curves in $\CrvList$ will be called the \myemph{cyclic index} of $\func$.
\end{definition}
It is easy to see that the cyclic index of $\func$ does not depend on a particular choice of a regular component $\curveMeridian$ of some level-set of $\func$ that does not separate $\Torus$.

\medskip

Let $\func|_{\Cyl_0}$ be the restriction of $\func$ onto $\Cyl_0$ and $\Orbit(\func|_{\Cyl_0}, \partial\Cyl_0)$ be the orbit of $\func|_{\Cyl_0}$ with respect to the action of the group $\Diff(\Cyl_0, \partial\Cyl_0)$ of diffeomorphisms of $\Cyl_0$ fixed on some neighbourhood of $\partial\Cyl_0$.
Now we can formulate the main result of the present paper.

\begin{theorem}\label{th:main:pi1Of}{\rm cf.~\cite{MaksymenkoFeshchenko:MS:2014}.}
Let $\func\in\FFF(\Torus)$ be such that $\fKRGraph$ has a cycle, $\curveMeridian$ be a regular connected component of certain level set $\func^{-1}(c)$ of $\func$ that does not separate $\Torus$, $\CrvList = \{\dif(\curveMeridian) \mid \dif\in\StabPrf \}$, and $n$ be the cyclic index of $\func$, i.e. the number of curves in $\CrvList$.

If $n=1$, then there is an isomorphism
\[
\xi: \pi_1\Orbit(\func) \ \cong \ \pi_1\Orbit(\func, \curveMeridian) \ \times \ \ZZZ.
\]
Suppose $n\geq2$ and let $\Cyl_0$ be the cylinder bounded by $\curveMeridian_0$ and $\curveMeridian_1$.
Then we have an isomorphism:
\[
\xi: \pi_1\Orbit(\func) \ \cong \ \pi_1\Orbit(\func|_{\Cyl_0}, \partial\Cyl_0) \ \wrm{n} \ \ZZZ.
\]
\end{theorem}

For $n=1$ this theorem is proved in~\cite{MaksymenkoFeshchenko:MS:2014}, therefore we will assume that $n\geq2$.

\subsection{Structure of the paper}
In \S\ref{sect:hom_types} we recall some results about the homotopy types of stabilizers and orbits of $\func\in\FFF(\Mman)$, and in \S\ref{sect:mult_in_rel_pi1} present some formulae for the multiplication in the relative homotopy group $\pi_1(D,S)$, where $D$ is a topological group and $S$ is its subgroup.

In \S\ref{sect:parallel_curves_on_t2} we consider families of parallel curves on $2$-torus and relations between Dehn twists and slides along these curves.
Given $\func\in\FFF(\Torus)$ such that its KR-graph has one cycle, we introduce in~\S\ref{sect:some_constructions} some special coordinates and flows adopted with $\func$.
In \S\ref{sect:two_epimorphisms} we define two epimorphisms $\eval:\pi_1(\DiffT,\StabPrf)\to\ZZZ$ and $\krot:\pi_0\StabPrf\to\ZZZ_n$ and study their properties, see Theorem~\ref{th:two_epimorphisms}.

As an interpretation of (c) Theorem~\ref{th:two_epimorphisms} we show in~\S\ref{sect:free_zn_action} that there exists a $\func$-invariant $\ZZZ_n$-action on $\Torus$, see Theorem~\ref{th:f_invar_zn_action}.
This interpretation is not used in the paper, but it gives a new view point of such functions $\func$.
Finally, in \S\ref{sect:proof:th:main:pi1Of} we complete Theorem~\ref{th:main:pi1Of}.

\section{Homotopy types of $\Stabf$ and $\Orbf$}\label{sect:hom_types}
Let $f\in\FFF(\Mman)$ and $\Xman$ be a finite (possibly empty) union of regular components of some level sets of $\func$.
We will briefly recall description of the homotopy types of $\StabfX$ and $\OrbfX$.

\begin{theorem}\label{th:serre_fibr}
{\rm\cite{Sergeraert:ASENS:1972, Maksymenko:AGAG:2006, Maksymenko:UMZ:ENG:2012}.}
The following map
\[
 \SerreFibr:\DiffMX \longrightarrow \OrbfX, \qquad \SerreFibr(h) = f \circ h.
\]
is a Serre fibration with fiber $\StabfX$, that is it has a homotopy lifting property for CW-complexes.

Hence $\SerreFibr(\DiffIdMX) = \OrbffX$ and the restriction map 
\begin{equation}\label{equ:fibr_pX}
\SerreFibr|_{\DiffIdMX}:\DiffIdMX \longrightarrow \OrbffX
\end{equation}
is also a Serre fibration with fiber $\StabPrfX=\Stabf\cap\DiffIdMX$;

Moreover, for each $k\geq0$ there is an isomorphism 
\[\fibr_k: \pi_k \bigl( \DiffMX, \StabfX \bigr) \to \pi_k \OrbfX\]
defined by $\fibr_k[\omega] = [f\circ\omega]$ for a continuous map $\omega:(I^k, \partial I^k, 0) \to \bigl( \DiffM, \Stabf, \id_{\Mman} \bigr)$, and making commutative the following diagram
\[
\xymatrix{ 
\cdots \ar[r] & \pi_k\DiffMX \ar[r]^-{\qhom} \ar[rd]_{p} & \pi_k\bigl(\DiffMX,\StabfX\bigr)\ar[d]^{\fibr_k}_{\cong} \ar[r]^-{\partial} & \pi_{k-1} \StabfX \ar[r] & \cdots \\
 & & \pi_k\OrbfX \ar[ru]_{\partial \circ \fibr_k^{-1}}
}
\]
see for example \cite[\S~4.1, Theorem~4.1]{Hatcher:AlgTop:2002}.
\end{theorem}

\begin{theorem}\label{th:fibration_DMX_Of}
{\rm \cite{Maksymenko:AGAG:2006, Maksymenko:ProcIM:ENG:2010, Maksymenko:UMZ:ENG:2012}.}
$\Orbit_{f}(f, \Xman) = \Orbit_{f}(f,\Xman\cup\partial\Mman)$, and so 
\[
\pi_k \Orbit(f, \Xman) \ \cong \ \pi_k \Orbit(f, \Xman \cup\partial\Mman), \qquad k\geq1.
\]
Suppose either $\func$ has a critical point which is not a {\bfseries nondegenerate local extremum} or $\Mman$ is a non-orientable surface.
Then $\StabIdf$ is contractible, $\pi_n\Orbf = \pi_n\Mman$ for $n\geq3$, $\pi_2\Orbf=0$, and for $\pi_1\Orbf$ we have the following short exact sequence of fibration $\SerreFibr$:
\begin{equation}\label{equ:pi1Of_exact_sequence}
 1 \longrightarrow \pi_1\DiffM \xrightarrow{~~\SerreFibr~~} \pi_1\Orbf \xrightarrow{~~\partial\circ \fibr^{-1}_1~~} \pi_0\StabPrf\longrightarrow 1.
\end{equation}
Moreover, $\SerreFibr\bigl(\pi_1\DiffM\bigr)$ is contained in the center of $\pi_1\Orbf$.

If either $\chi(\Mman)<0$ or $\Xman\not=\varnothing$.
Then $\DiffIdMX$ and $\StabIdfX$ are contractible, whence from the exact sequence of homotopy groups of the fibration~\eqref{equ:fibr_pX} we get $\pi_k\OrbfX=0$ for $k\geq2$, and that the boundary map 
\[ \partial\circ \fibr^{-1}_1: \pi_1\OrbfX \ \longrightarrow \ \pi_0 \StabPrfX \]
is an isomorphism.
\end{theorem}

Suppose $\Mman$ is differs from $2$-sphere $S^2$, $2$-torus, projective plane, and Klein bottle, and let $\Xman = \partial\Mman$.
Then $\Mman$ and $\Xman$ satisfy assumptions of Theorem~\ref{th:fibration_DMX_Of}, and we get the following isomorphisms
\[
\pi_1\Orbf \ \cong \ \pi_1\Orbit(f, \partial\Mman) \ \cong \ \pi_0 \Stab'(f,\partial\Mman).
\]
A possible structure of $\pi_0\Stab'(f,\partial\Mman)$ for this case is completely described in~\cite{Maksymenko:pi1Repr:2014}.

However in the remained four cases of $\Mman$ we have that $\pi_1\DiffM \not=0$, and all terms in the short exact sequence~\eqref{equ:pi1Of_exact_sequence} can be non-trivial.

In particular, suppose $\Mman=\Torus$.
Then the sequence~\eqref{equ:pi1Of_exact_sequence} has the following form:
\begin{equation}\label{equ:T2_pi1Of_exact_sequence}
1 \longrightarrow \ZZZ^2 \xrightarrow{~~p~~} \pi_1\Orbff \xrightarrow{~~\partial~~} \pi_0 \StabPrf \longrightarrow 1.
\end{equation}
It is shown in~\cite{MaksymenkoFeshchenko:UMZ:ENG:2014} that if a KR-graph $\fKRGraph$ of $\func\in\FFF(\Torus)$ is a tree, then under some additional ``triviality'' assumption on the action $\StabPrf$ on $\fKRGraph$, the sequence~\eqref{equ:T2_pi1Of_exact_sequence} splits.

Moreover, in~\cite{MaksymenkoFeshchenko:MS:2014} the authors considered the case when $\fKRGraph$ of $\func\in\FFF(\Torus)$ has one cycle, and $\func$ has cyclic index $n=1$.

\section{Multiplication in $\pi_1(D,S,e)$}\label{sect:mult_in_rel_pi1}
Let $D$ be a topological space, $S$ be its subset, and $e\in S$ be a point.
Then, in general, the relative homotopy set $\pi_1(D,S,e)$, as well as $\pi_0 (D,e)$ and $\pi_0(S,e)$ \myemph{\bfseries have no natural group structure}.
However, if $D$ is a topological group, $S$ is a subgroup of $D$, and $e$ is the unit of $D$, then such group structures exist.
We leave the following lemma for the reader.
\begin{lemma}\label{lm:DSe_group_structures}{\rm cf.~\cite[Ch~1, \S4]{FomenkoFuks:HomotTopology:1989}}
Let $D$ be a topological group with multiplication $\circ$, $S$ be a subgroup of $D$, and $e$ be the unit of $D$.
Then $\pi_0 (D,e)$, $\pi_1(D,S,e)$, $\pi_0 (S,e)$ have a group structures such that in the corresponding sequence of homotopy groups of the triple $(D,S,e)$:
\[
\cdots \to \pi_1 (D,e) \xrightarrow{~\qhom~} \pi_1 (D,S,e) \xrightarrow{~\partial~} \pi_0 (S,e) \xrightarrow{~i~} \pi_0 D \to \cdots
\]
the maps $\qhom$, $\partial$, and $i$ are homomorphisms.
Moreover $\qhom(\pi_1(D,e))$ is contained in the centre of $\pi_1 (D,S,e)$.
\qed
\end{lemma}

In what follows we will assume that $D$, $S$, and $e$ are the same as in Lemma~\ref{lm:DSe_group_structures}.
We will recall a formula for the multiplication in $\pi_1(D,S,e)$.

Let $\gdif,\dif:(I,\partial I, 0) \to (D,S,e)$ be two paths representing some elements of $\pi_1(D,S,e)$.
For simplicity we will denote $\gdif(t)$ by $\gdif_t$ and similarly for $\dif$.
The class of $[\gdif]\in\pi_1(D,S,e)$ will also be denoted by $[\gdif_t]$.
Define another path $r:(I,\partial I, 0) \to (D,S,e)$ by
\[
r(t) =
\begin{cases}
\gdif_{2t}, & t\in[0,\tfrac{1}{2}], \\
\gdif_1 \circ \dif_{2t-1}, & t\in[\tfrac{1}{2},1].
\end{cases}
\]
Then $[r_t] = [\gdif_t]\pmult[\dif_t]$ in $\pi_1(D,S,e)$.

As an immediate consequence of this formula we get the following lemma:
\begin{lemma}\label{lm:prod_in_rel_pi1}
Let $\gdif,\dif:I \to D$ be two paths such that $\gdif(0)=e$, $\gdif(1)=\dif(0) \in S$ and $\dif(1)\in S$ as well, and $s:(I,\partial I, 0) \to (D,S,e)$ be a path defined by 
\[
s(t) =
\begin{cases}
\gdif_{2t}, & t\in[0,\tfrac{1}{2}], \\
\dif_{2t-1}, & t\in[\tfrac{1}{2},1],
\end{cases}
\]
so it is obtained by joining $\gdif$ and $\dif$, see Figure~\ref{fig:paths_in_the_square}(a).
Then 
\begin{equation}\label{equ:path_pull_back}
[s_t] = [\gdif_t]\pmult[\gdif_1^{-1}\circ\dif_t]
\end{equation}
 in $\pi_1(D,S,e)$, where $[\gdif_1^{-1}\circ\dif_t]$ is a class of a path $(I,\partial I, 0) \to (D,S,e)$ defined by $t\mapsto \gdif_1^{-1}\circ\dif_t$.
\qed
\end{lemma}

\begin{lemma}\label{lm:relations_in_rel_pi1}
Let $\gdif_t, \dif_t:(I,\partial I, 0) \to (D,S,e)$ be two paths.
Then in $\pi_1(D,S,e)$ we have the following identities:
\begin{gather}
\label{equ:mult_pi1Of_prod}
[\gdif_t\circ\dif_t] = 
[\gdif_s]\pmult[\dif_t] =
[\dif_t]\pmult[\dif_1^{-1}\circ\gdif_s \circ\dif_1],
 \\
\label{equ:mult_pi1Of_conj}
[\dif_t]\pmult[\gdif_s]\pmult[\dif_t^{-1}] = [\dif_1^{-1}\circ\gdif_s\circ\dif^{-1}],
\end{gather}
where $[\gdif_t\circ\dif_t]$ means the class of the path $(I,\partial I, 0) \to (D,S,e)$ given by $t\mapsto\gdif_t\circ\dif_t$, and similarly for all other classes.
\end{lemma}
\begin{proof}
Let $\gamma:I\times I \to D$ be a continuous map defined by
\[
\gamma(s,t) = \gdif_s\circ\dif_t, \qquad (s,t)\in I\times I,
\]
see Figure~\ref{fig:paths_in_the_square}(b).
\begin{figure}[h]
\begin{tabular}{ccc}
\includegraphics[height=2.4cm]{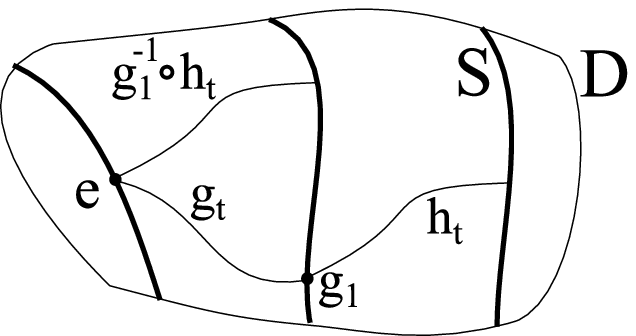} & \qquad\qquad &
\includegraphics[height=2.4cm]{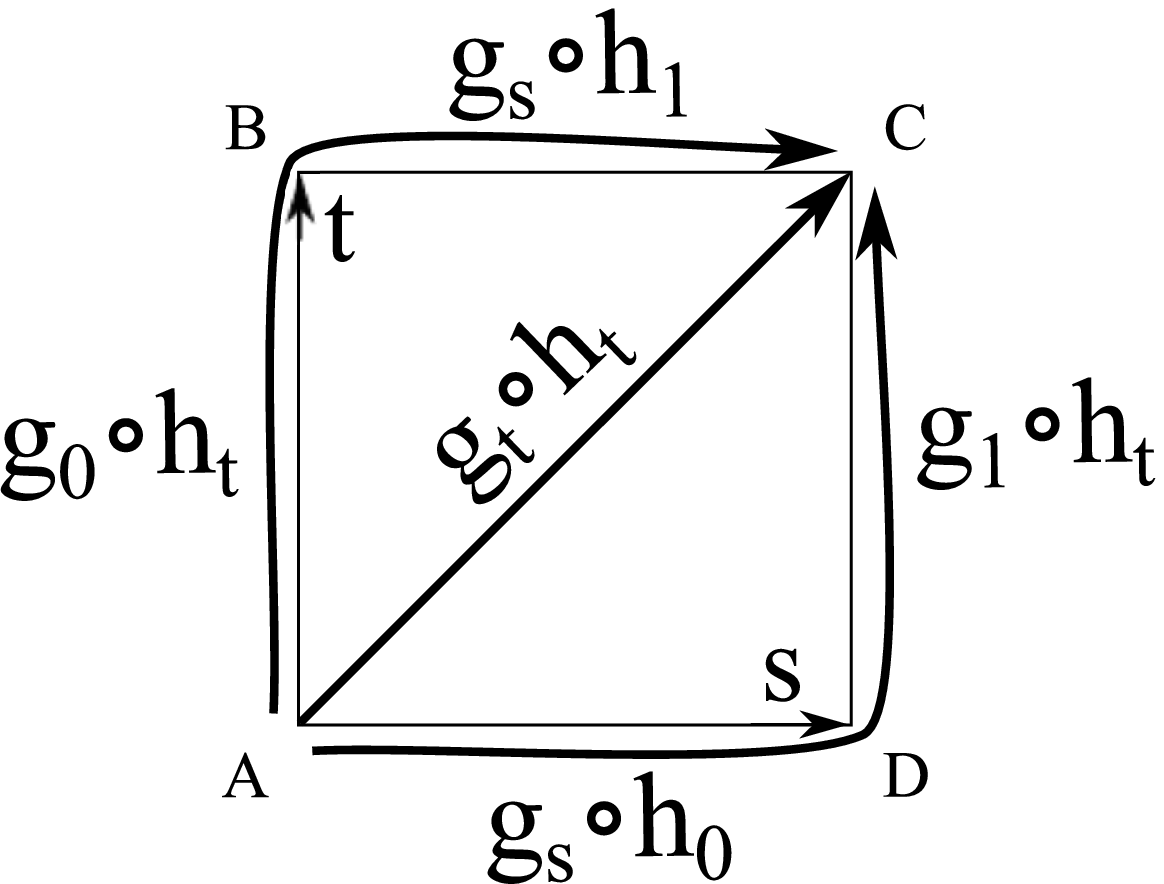} \\
(a) && (b)
\end{tabular}
\caption{}\label{fig:paths_in_the_square}
\end{figure}
Then the path $[\gdif_t\circ\dif_t]$ corresponds to the restriction of $\gamma$ to the diagonal $AC = \{s=t \mid (s,t)\in I\times I\}$.
Evidently, this path is homotopic relatively to its ends to the composition of paths along sides $AB$ and $BC$ as well as to the composition of paths along sides $AD$ and $DC$.
Hence by~\eqref{equ:path_pull_back} we get the following relations in $\pi_1(D,S,e)$:
\begin{align*}
[\gdif_t\circ\dif_t] &= 
[\gdif_s\circ\dif_0]\pmult[(\gdif_1\circ\dif_0)^{-1}\circ \gdif_1\circ\dif_t] =
[\gdif_s]\pmult[\dif_t], \\
[\gdif_t\circ\dif_t] &= 
[\gdif_0\circ\dif_t]\pmult[(\gdif_0\circ\dif_1)^{-1}\circ\gdif_s\circ\dif_1] =
[\dif_t]\pmult[\dif_1^{-1}\circ \gdif_s\circ\dif_1], \\
[\dif_t]\pmult[\gdif_s]\pmult[\dif_t^{-1}] 
&= [\dif_t]\pmult[\dif_t^{-1}]\pmult[\dif_1\circ\gdif_s\circ\dif_1^{-1}] = [\dif_t \circ\dif_t^{-1}]\pmult[\gdif_s\circ\dif_1^{-1}] = [\gdif_s\circ\dif_1^{-1}],
\end{align*}
where we take to account that $\gdif_0=\dif_0=e$.
\end{proof}

\section{Parallel curves on $\Torus$}\label{sect:parallel_curves_on_t2}

\subsection{Twists and slides along curves}\label{sect:twists_and_slides}
Let $\alpha,\beta:[-1,1]\to[0,1]$ be two $C^{\infty}$-functions such that $\alpha=0$ on $[-1,-\tfrac{1}{2}]$ and $\alpha=1$ on $[\frac{1}{2},1]$, while $\beta=0$ on $[-1,-\tfrac{2}{3}] \cup [\tfrac{2}{3},1]$ and $\beta = 1$ on $[-\tfrac{1}{3}, \tfrac{1}{3}]$, see Figure~\ref{fig:twist_and_slide}.

Let also $\Cyl = S^1\times [-1,1]$ be a cylinder and $\curveMeridian = S^1\times 0$.
Define the following two diffeomorphisms $\dtw, \slide:\Cyl\to\Cyl$ by 
\begin{align*}
\dtw(z,t) &= ( z e^{\alpha(t)}, t),
&
\slide(z,t) &= ( z e^{\beta(t)}, t),
\end{align*}
for $(z,t)\in \Cyl$, see Figure~\ref{fig:twist_and_slide}.
Then $\dtw$ is called a \myemph{Dehn twist} and $\slide$ is called a \myemph{slide} along the curve $\curveMeridian$.
Notice that $\dtw$ is fixed on some neighbourhood of $\partial\Cyl$, while $\slide$ is fixed on some neighbourhood of $\curveMeridian \cup \partial\Cyl$.

\begin{figure}[h]
\begin{tabular}{ccc}
\includegraphics[height=1.4cm]{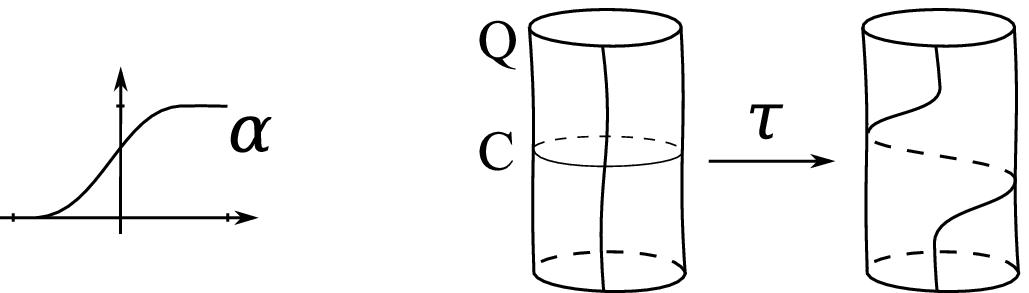}  & \qquad\qquad\qquad & \includegraphics[height=1.4cm]{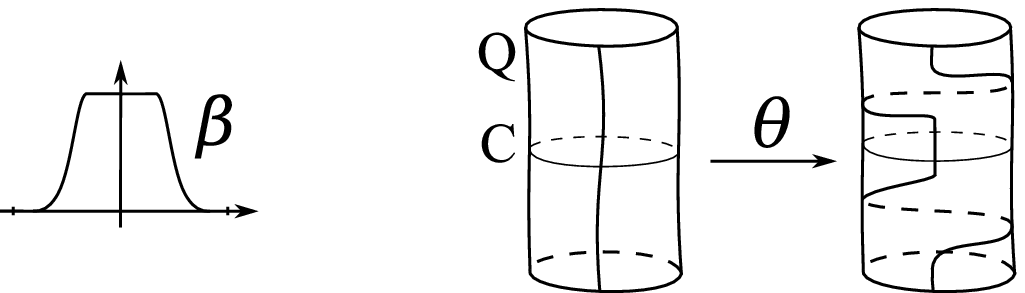} \\
(a) Dehn twist & & (b) Slide
\end{tabular}
\caption{}\label{fig:twist_and_slide}
\end{figure}

\begin{lemma}\label{lm:DehnTwist_gen_DQdQ}
Let $\Diff(\Cyl,\partial\Cyl)$ be the group of diffeomorphisms fixed on some neighbourhood of $\partial\Cyl = S^1\times\{0,1\}$, and $\dtw \in \Diff(\Cyl,\partial\Cyl)$ be a Dehn twist along the curve $\curveMeridian$.
Then \[ \pi_0\Diff(\Cyl,\partial\Cyl)= \langle[\dtw]\rangle \cong \ZZZ,\] i.e. it is an infinite cyclic group generated by the isotopy class of the Dehn twist $\tau$.
\end{lemma}

Now let $\curveMeridian \subset \Mman$ be a simple closed curve.
Suppose $\curveMeridian$ \myemph{preserves orientation}, that is it has a closed neighbourhood $\CNbh$ diffeomorphic to a cylinder $\Cyl$.
Fix any $\phi:\Cyl\to\CNbh$ such that $\phi(S^1\times0) = \curveMeridian$.

Since $\dtw$ is fixed on some neighbourhood of $\partial\Cyl$, we see that $\phi\circ\dtw\circ\phi^{-1}:\CNbh\to\CNbh$ extends by the identity to some diffeomorphism $\bar{\dtw}$ and $\bar{\slide}$ of $\Mman$ respectively.
Any diffeomorphism $\dif:\Mman\to\Mman$ isotopic to $\bar{\dtw}$ or $\bar{\dtw}^{-1}$ will be called a \myemph{Dehn twist} along $\curveMeridian$.

Also notice that $\slide$ is fixed on some neighbourhood of $(S^1\times0) \cup \partial\Cyl$, whence the diffeomorphism  $\phi\circ\slide\circ\phi^{-1}:\CNbh\to\CNbh$ extends by the identity to some diffeomorphisms $\bar{\slide}$ of $\Mman$.
Any diffeomorphism $\dif:\Mman\to\Mman$ \emph{fixed on some neighbourhood of $\curveMeridian$}, supported in some cylindrical neighbourhood $\CNbh$ of $\curveMeridian$, and isotopic to $\bar{\slide}$ or $\bar{\slide}^{-1}$ relatively to some neighbourhood of $\curveMeridian \cup \overline{\Mman\setminus\Cyl}$ will be called a \myemph{slide} along $\curveMeridian$.

\subsection{Diffeomorphisms of $\Torus$ fixed on parallel family of curves}
Let $\curveMeridian_0,\ldots,\curveMeridian_{n-1} \subset \Torus$ be a parallel family of curves cyclically ordered along $\Torus$, see~\S\ref{sect:paral_curves} and Figure~\ref{fig:paral_curves}.
For each $i=0,\ldots,n-1$ let $\dtw_i\in\Diff(\Torus)$ be a Dehn twist such that $\supp(\dtw_i) \subset \Int{\Cyl_i}$ and its restriction $\dtw_i|_{\Cyl_i}$ generates $\pi_0\Diff(\Cyl_i,\partial\Cyl_i) \cong \ZZZ$, see Figure~\ref{fig:twist_and_slide_on_t2}(a).
Replacing, if necessary, $\dtw_i$ with $\dtw_i^{-1}$ we can assume that \myemph{all $\dtw_i$ are isotopic each other as diffeomorphisms of $\Torus$}.

\begin{figure}[h]
\begin{tabular}{ccc}
\includegraphics[height=2cm]{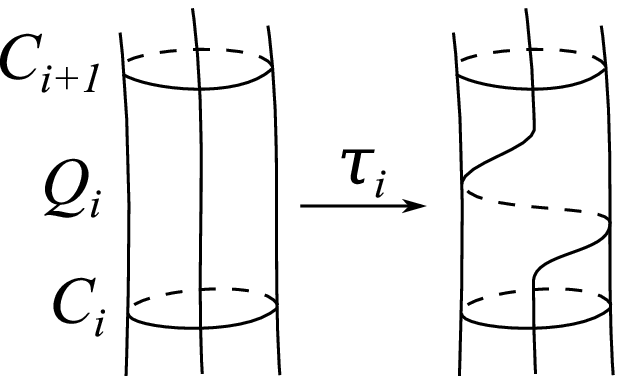}  & \qquad\qquad\qquad & \includegraphics[height=2cm]{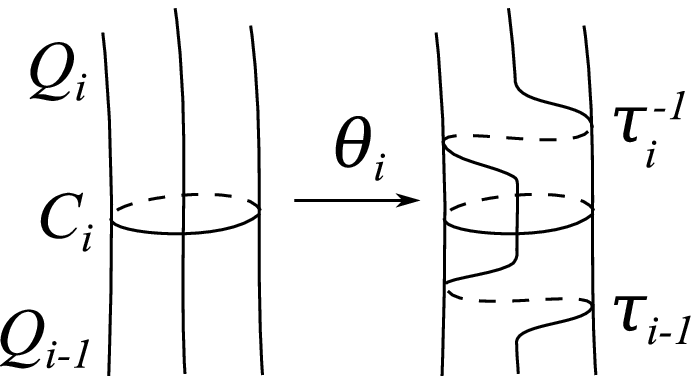} \\
(a) Dehn twist & & (b) Slide
\end{tabular}
\caption{}\label{fig:twist_and_slide_on_t2}
\end{figure}

Let 
\begin{equation}\label{equ:group_of_slides}
\grp = \DiffId(\Torus) \cap \Diff(\Torus,\CrvList)
\end{equation}
be the group of diffeomorphisms fixed on some neighbourhood of each $\curveMeridian_i$ and isotopic to the identity via an isotopy that is not necessarily fixed near $\CrvList$.
Evidently, $\DiffId(\Torus,\CrvList)$ is the path component of $\grp$ containing $\idT$, whence
\[
\pi_0 \grp \ \cong \ \grp / \DiffId(\Torus,\CrvList).
\]

\begin{theorem}\label{th:slides_gen_pi0G}
Let $\slide_i\in\grp$, $i=0,\ldots,n-1$, be a slide along $\curveMeridian_i$ such that
\begin{enumerate}
\item[\rm(i)]
$\supp(\slide_i) \subset \Int{\Cyl_{i-1}} \cup \Int{\Cyl_{i}}$, and, in particular, $\slide_i$ is fixed near $\Cyl_i$;
\item[\rm(ii)]
$\supp(\slide_i)\cap\supp(\slide_j)=\varnothing$ for $i\not=j$;
\item[\rm(iii)]
$\slide|_{\Cyl_i}$ is isotopic to $\dtw_{i-1}\circ\dtw_{i}^{-1}$ relatively to some neighbourhood of $\curveMeridian_i \cup \Mman\setminus(\Cyl_{i-1}\cup\Cyl_i)$, see Figure~\ref{fig:twist_and_slide_on_t2}(b).
\end{enumerate} 
Denote $\slide = \slide_0\circ\slide_1\circ\cdots\circ\slide_{n-1}$.
Then $\slide \in \DiffId(\Torus,\CrvList)$, i.e. it is isotopic to $\idT$ relatively to some neighbourhood of $\CrvList$.
Moreover,
\begin{equation}\label{equ:pi0G_gen_by_slides}
\pi_0 \grp \ \cong \ \langle[\slide_1],\ldots,[\slide_{n-1}] \rangle \ \cong \ \ZZZ^{n-1},
\end{equation}
i.e. this group is freely generated by isotopy classes of slides $\slide_1,\ldots,\slide_{n-1}$ in $\grp$.

In particular, if $n=1$, $\pi_0 \grp=\{1\}$, and so $\grp = \DiffId(\Torus) \cap \Diff(\Torus,\CrvList) = \DiffId(\Torus,\CrvList)$.
\end{theorem}
\begin{proof}
For $n=1$ this statement is established in~\cite{MaksymenkoFeshchenko:MS:2014}, therefore we will assume that $n\geq2$.

It follows from (iii) that $\slide$ is isotopic relatively to some neighbourhood of $\CrvList$ to 
\[
\dtw_{0}\circ\dtw_{1}^{-1} \circ \dtw_{1}\circ\dtw_{2}^{-1} \circ \cdots \circ \dtw_{n-1} \circ \dtw_{0}^{-1} = \idT,
\]
that is $\slide\in \DiffId(\Torus,\CrvList)$.

It remains to prove~\eqref{equ:pi0G_gen_by_slides}.
Evidently, if $\dif\in\grp$, then $\dif(\Cyl_i)=\Cyl_i$ and $\dif$ is fixed on some neighbourhood of $\partial\Cyl_i = \curveMeridian_i \cup \curveMeridian_{i+1}$.
In other words, the restriction $\dif|_{\Cyl_i} \in \Diff(\Cyl_i,\partial\Cyl_i)$.
Hence, by Lemma~\ref{lm:DehnTwist_gen_DQdQ}, $\dif|_{\Cyl_i}$ is isotopic relatively to some neighbourhood $\partial\Cyl_i$ to $\dtw_i^{a_i}|_{\Cyl_i}$ for a unique $a_i\in\ZZZ$.
Therefore $\dif$ itself is isotopic relatively to some neighbourhood of $\CrvList$ to the product 
\begin{equation}\label{equ:prod_dtw}
\dtw_0^{a_0} \circ \dtw_1^{a_1} \circ \cdots \circ \dtw_{n-1}^{a_{n-1}}
\end{equation}
for unique integers $a_0,\ldots,a_{n-1}\in\ZZZ^n$.

It easily follows that the correspondence $\dif\longmapsto (a_0,\ldots,a_{n-1})$ is a well-defined homomorphism
\[
\qiso:\grp \longrightarrow \ZZZ^{n}.
\]
Consider the following subgroup of $\ZZZ^n$:
\[\Delta = \{(a_0,\ldots,a_{n-1}) \in \ZZZ^n \mid a_0+\cdots+a_{n-1}=0 \}.\]
\begin{lemma}
$\ker(\qiso) = \DiffId(\Torus,\CrvList)$ \ and \ $\qiso(\grp) = \Delta$, so we have the following exact sequence
\[
1 \longrightarrow \DiffId(\Torus,\CrvList) \xrightarrow{~~\subset~~}  \grp \xrightarrow{~~\qiso~~} \Delta \longrightarrow 1.
\]
Hence $\pi_0\grp \cong \grp/\DiffId(\Torus,\CrvList)  \cong \Delta \cong \ZZZ^{n-1}$.
\end{lemma}
\begin{proof}
The identity $\ker(\qiso) = \DiffId(\Torus,\CrvList)$ easily follows from Lemma~\ref{lm:DehnTwist_gen_DQdQ}.

Let us prove that $\qiso(\grp) = \Delta$.
Suppose $\qiso(\dif) = (a_0,\ldots,a_{n-1})$, so $\dif$ is isotopic relatively to some neighbourhood of $\CrvList$ to the product $\dtw_0^{a_0} \circ \dtw_1^{a_1} \circ \cdots \circ \dtw_{n-1}^{a_{n-1}}$.
But by construction all $\dtw_i$ are mutually isotopic as diffeomorphisms of $\Torus$.
Hence $\dif$ is isotopic to $\dtw_0^{a_0+\cdots+a_{n-1}}$.
On the other hand, by assumption $\dif$ is isotopic to $\idT$, while $\dtw_0$ is not isotopic to the identity and its isotopy class in $\pi_0\Diff(\Torus)$ has infinite order.
Therefore $a_0+\cdots+a_{n-1} = 0$, i.e. $\qiso(\dif)\in\Delta$.
\end{proof}

Now we can complete the proof of Theorem~\ref{th:slides_gen_pi0G}.
By (ii) $\slide_i$ is isotopic relatively $\CrvList$ to the product $\dtw_{i-1} \circ \dtw_{i}^{-1}$, see Figure~\ref{fig:twist_and_slide_on_t2}(b).
This means that
\[
\qiso(\slide_i) = (\underbrace{0,\ldots,0, 1}_{i}, -1, 0,\ldots,0),
\qquad 
i=1,\ldots,n-1.
\]
It remains to note that the elements $\qiso(\slide_i)$, $i=1,\ldots,n-1$, constitute a basis for $\Delta$, whence their isotopy classes in $\grp$ constitute a basis for $\pi_0\grp$.
\end{proof}

\begin{corollary}\label{cor:slides_gen_pi0G}
For each $\dif\in\grp$ there exist unique $b_1,\ldots,b_{n-1}\in\ZZZ$ and $\gdif\in\DiffIdTC$ such that
$\dif = \slide_1^{b_1} \circ \cdots \circ \slide_{n-1}^{b_{n-1}} \circ \gdif$.
\qed
\end{corollary}

\subsection{Smooth shifts along trajectories of a flow}
Let $\flow:M \times \RRR\to M$ be a smooth flow on a manifold $M$.
Then for every smooth function $\alpha:M \to \RRR$ one can define the following map $\flow_{\alpha}:\Torus\to\RRR$ by the formula:
\begin{equation}\label{equ:shooth_shift}
\flow_{\alpha}(z) = \flow(z, \alpha(z)), \qquad z\in M. 
\end{equation}
\begin{lemma}\label{lm:isotopy_between_shifts}
{\rm\cite[Claim~4.14.1]{Maksymenko:AGAG:2006}.}
Suppose $\flow_{\alpha}$ is a {\bfseries diffeomorphism}.
Then for each $t\in[0,1]$ the map
\begin{align*}
&\flow_{t\alpha}:M \to M,
&
\flow_{t\alpha}(z) &= \flow(z, t\alpha(z))
\end{align*}
is a diffeomorphism as well.

In particular, $\{\flow_{t\alpha}\}_{t\in I}$ is an isotopy between $\id_{M} = \flow_0$ and $\flow_{\alpha}$.
\qed
\end{lemma}

\section{Some constructions associated with $\func$}\label{sect:some_constructions}
In the sequel we will regard the circle $S^1$ and the torus $\Torus$ as the corresponding factor-groups $\RRR/\ZZZ$ and $\RRR^2/\ZZZ^2$.
For $s\in S^1$ and $\eps\in(0,0.5)$ let \[ J_{\eps}(s) = (s-\eps,s+\eps) \subset S^1\]
be an open $\eps$-neighbourhood of $s\in S^1$.

\medskip 

Let $f\in\FFF(\Torus)$ be a function such that its KR-graph $\fKRGraph$ has only one cycle, $\curveMeridian$ be a regular connected component of certain level set of $\func$ not separating $\Torus$, and 
\[ \CrvList = \{\dif(\curveMeridian) \mid \dif\in\StabPrf \} = \{ 
\curveMeridian_0=\curveMeridian, \ \curveMeridian_1, \ \ldots,  \ \curveMeridian_{n-1} \},
\]
see Figure~\ref{fig:f_curves}.
We will now define several constructions ``adopted'' with $\func$.

\subsection*{Special coordinates.}
As the curves $\{\curveMeridian_i \mid i=0,\ldots,n-1\}$ are ``parallel'', one can assume (by a proper choice of coordinates on $\Torus$) that the following two conditions hold:
\begin{enumerate}
\item[\rm(a)]
$\curveMeridian_i = \frac{i}{n} \times S^1 \ \subset \ \RRR^2/\ZZZ^2 \equiv \Torus$\,;
\item[\rm(b)] 
there exists $\eps>0$ such that for all $t \in J_{\eps}(\tfrac{i}{n}) = (\tfrac{i}{n}-\eps,\tfrac{i}{n} +\eps)$ the curve $t \times S^1$ is a regular connected component of some level set of $\func$.
\end{enumerate}
It is convenient to regard each $\curveMeridian_k$ as a \emph{meridian} of $\Torus$.
Let $\curveParallel = S^1 \times 0$ be the corresponding \emph{parallel}.
Then $\curveParallel \cap \curveMeridian_{i}  = \tfrac{i}{n}$. 

\subsection*{Isotopies $\flowParallel$ and $\flowMeridian$.}
Let $\flowParallel, \flowMeridian: \Torus \times [0,1]\to \Torus$ be two isotopies defined by
\begin{align}\label{equ:LM_flows}
\flowParallel(x,y,t) &= (x + t \ \mathrm{mod} \ 1,\  y),
&
\flowMeridian(x,y,t) &= (x, \ y + t \ \mathrm{mod} \ 1),
\end{align}
for $x\in \curveParallel$, $y\in \curveMeridian$, and $t\in[0,1]$.
Geometrically, $\flowParallel$ is a ``\emph{rotation}'' of the torus along its parallels and $\flowMeridian$ is a rotation along its meridians.
We can regard them as loops in $\pi_1\DiffT$.
Denote by $\subgroupParallel$ and $\subgroupMeridian$ the subgroups of $\pi_1\DT$ generated by loops $\flowParallel$ and $\flowMeridian$ respectively.
It is well known that that 
$\subgroupParallel$ and $\subgroupMeridian$ are commuting free cyclic groups, and so we get an isomorphism:
\[
\pi_1\DT \cong \subgroupParallel \times \subgroupMeridian.
\]

Also notice that $\flowParallel$ and $\flowMeridian$ can be also regarded as \emph{flows} $\flowParallel, \flowMeridian: \Torus \times \RRR\to \Torus$ defined by the same formulas Eq.~\eqref{equ:LM_flows} for $(x,y,t)\in\Torus\times\RRR$.
All orbits of the \textit{flows} $\flowParallel$ and $\flowMeridian$ are periodic of period $1$.

\subsection*{A flow $\flow$.}
Since $\Torus$ is an orientable surface, one can construct a ``Hamiltonian like'' flow $\flow:\Torus\times\RRR\to\Torus$ having the following properties, see e.g.~\cite[Lemma~5.1]{Maksymenko:AGAG:2006}:
\begin{itemize}
\item[1)]
a point $z\in\Torus$ is fixed for $\flow$ if and only if $z$ is a critical point of $\func$;
\item[2)]
$\func$ is constant along orbits of $\flow$, that is $f(z) = f(\flow(z,t))$ for all $z\in\Torus$ and $t\in\RRR$.
\end{itemize}
It follows that every critical point of $\func$ and every regular components of every level set of $\func$ is an orbit of $\flow$.

In particular, each curve $t\times S^1$ for $t\in J_{\eps}(\tfrac{i}{n})$, $i=0,\ldots,n-1$, is an orbit of $\flow$.
On the other hand, this curve is also an orbit of the flow $\flowMeridian$.
Therefore, we can always choose $\flow$ so that
\begin{equation}\label{equ:relation_flowMer_flowHam}
\flowMeridian(x,y,t) = \flow(x,y,t),
\end{equation}
for $(x,y,t) \in J_{\eps}(\tfrac{i}{n})\times S^1\times\RRR$ and $i=0,\ldots,n-1$.

\begin{lemma}\label{lm:shift_functions}{\rm\cite{Maksymenko:AGAG:2006, Maksymenko:ProcIM:ENG:2010}}.
Suppose a flow $\flow:\Torus\times\RRR\to\Torus$ satisfies the above conditions {\rm1)} and {\rm2)}.
Then the following statements hold.

\medskip 
{\rm(1)}~Let $\dif\in\Stabf$.
Then $h\in\StabIdf$ if and only if there exists a $C^{\infty}$ function $\alpha:\Torus\to\RRR$ such that $\dif = \flow_{\alpha}$, see~\eqref{equ:shooth_shift}.
Such a function is unique and the family of maps $\{\flow_{t\alpha}\}_{t\in I}$ constitute an isotopy between $\id_{M}$ and $\dif$, \cite[Lemma~16]{Maksymenko:ProcIM:ENG:2010}.

\medskip 
{\rm(2)}~Suppose $\curveMeridian$ is a regular component of some level set of $\func$ and $\dif\in\Stabf$ be such that $\dif(\curveMeridian)=\curveMeridian$ and $\dif$ preserves orientation of $\curveMeridian$.
Let also $\Nman$ be an arbitrary open neighbourhood of $\curveMeridian$.
Then each $\dif\in\Stabf$ is isotopic in $\Stabf$ via an isotopy supported in $\Nman$ to a diffeomorphism $\gdif$ fixed on some smaller neighbourhood of $\curveMeridian$.
In particular, $[\dif]=[\gdif]\in\pi_0\Stabf$, \cite[Lemma~4.14]{Maksymenko:AGAG:2006}.

\medskip 
{\rm(3)}~Let $\Xman$ be a finite disjoint union of regular components of some level sets of $\func$, and $\Nman$ be an open neighbourhood of $\Xman$.
Then there exists a smaller open neighbourhood $\Uman\subset\Nman$ of $\Xman$ such that $\overline{\Uman}\subset\Nman$ and each  $\dif\in\StabIdf$ is isotopic in $\Stabf$ relatively to $\overline{\Uman}$ to a diffeomorphism $\gdif$ fixed on $\Mman\setminus\Nman$.
In particular, $\gdif\in\StabIdf$ as well.
Moreover, if $\dif = \flow_{\alpha}$, then one can assume that $\gdif = \flow_{\beta}$, where $\beta=\alpha$ on $\Uman$ and $\beta=0$ on $\Mman\setminus\Nman$, \cite[Lemma~4.14]{Maksymenko:AGAG:2006}.
\qed
\end{lemma}

\subsection*{Special slides.}
It follows from~\eqref{equ:relation_flowMer_flowHam} and~\eqref{equ:LM_flows} that each $\curveMeridian_k$ is an orbit of the flow $\flow$ of period $1$.
Let $\alpha, \beta:[-1,1]\to[0,1]$ be the functions defined in~\S\ref{sect:twists_and_slides}, see Figure~\ref{fig:twist_and_slide}, and $\eps$ be the same as in~\eqref{equ:relation_flowMer_flowHam}. 
Define two diffeomorphisms $\dtw_i, \slide_i:\Torus\to\Torus$, $i=0,\ldots,n-1$, by the formulae:
\begin{equation}\label{equ:special_twists}
\dtw_i(x,y) =
\begin{cases}
\flow\bigl(\,x,y, \alpha((y - \tfrac{i}{n})/2\eps) \,\bigr), & (x,y)\in J_{\eps}(\tfrac{i}{n})\times S^1, \\
(x,y), & \text{otherwise},
\end{cases}
\end{equation}
\begin{equation}\label{equ:special_slides}
\slide_i(x,y) =
\begin{cases}
\flow\bigl(\,x,y, \beta((y - \tfrac{i}{n})/2\eps) \,\bigr), & (x,y)\in J_{\eps}(\tfrac{i}{n})\times S^1, \\
(x,y), & \text{otherwise}.
\end{cases}
\end{equation}
Evidently, $\dtw_i$ is a Dehn twist and $\slide_i$ is a \myemph{slide} along $\curveMeridian_i$ in the sense of~\S\ref{sect:twists_and_slides}.

Notice that $\func\circ\slide_i=\func$, $\slide_i$ is isotopic to $\idT$, and $\slide_k$ is also fixed on some neighbourhood of $\CrvList$.
In other words, 
\[\slide_i \ \in \ \Stabf \cap \DiffIdT \cap \DiffTC \ = \ \Stabf\cap\grp,\]
see~\eqref{equ:group_of_slides}.
Moreover, $\supp(\slide_i) \cap \supp(\slide_j)=\varnothing$ for $i\not=j \in\{1,\ldots,n-1\}$.
Let also
\begin{equation}\label{equ:all_slides_product}
\slide = \slide_0 \circ \cdots \circ \slide_{n-1}.
\end{equation}
Then by Theorem~\ref{th:slides_gen_pi0G} $\slide\in\Stabf \cap \DiffIdTC = \SfC$.
Let $\cclass{\slide}$ be the isotopy class of $\slide$ in $\pi_0\SfC$, and $\Slide = \langle \cclass{\slide} \rangle$ be the subgroup of $\pi_0\SfC$ generated by $\cclass{\slide}$.

The following lemma is an easy consequence of~\eqref{equ:special_slides} and~\eqref{equ:all_slides_product} and we leave it for the reader.
\begin{lemma}\label{lm:slide_properties}
$\slide = \flow_{\sigma} = \flowMeridian_{\sigma}$ for some $C^{\infty}$ function $\sigma$ such that $\sigma=1$ on $\CrvList$.
Moreover, as $\sigma$ is constant along orbits of $\flow$, it follows from~\cite[Eq.~(8)]{Maksymenko:TA:2003} and can easily be shown, that $\slide^{k} = \flow_{k\sigma}$ for all $k\in\ZZZ$.
\qed
\end{lemma}

\section{Two epimorphisms}\label{sect:two_epimorphisms}
In the notation of~\S\ref{sect:some_constructions} let $\func\in\FFF(\Torus)$ be such that its KR-graph $\fKRGraph$ has exactly one cycle, $\curveMeridian$ be a regular connected component of certain level set $\func^{-1}(c)$ of $\func$ that does not separate $\Torus$,
\[ \CrvList = \{\dif(\curveMeridian) \mid \dif\in\StabPrf \}\]
be the corresponding family of curves parallel to $\curveMeridian$, and $n$ be the number of curves in $\CrvList$.
The case $n=1$ is considered in~\cite{MaksymenkoFeshchenko:MS:2014}, therefore we will assume that $n\geq1$.

For simplicity we will introduce the following notation:
\begin{align*}
\DT &:= \DiffId(\Torus), &
\Of &:= \Orbit_{\func}(\func), &
\Sf &:= \Stab'(\func), &
\Sidf &:= \StabId(\Torus), \\
\DTC &:= \DiffId(\Torus, \CrvList), &
\OfC &:= \Orbit_{\func}(\func,\CrvList), &
\SfC &:= \Stab'(\func,\CrvList), &
\SidfC &:= \StabId(\func,\CrvList), \\
\DQ &:= \DiffId(\Cyl_0, \partial\Cyl_0), &
\OfQ &:= \Orbit(\func|_{\Cyl_0}, \partial\Cyl_0), &
\SfQ &:=\Stab(\func|_{\Cyl_0}, \partial\Cyl_0).
\end{align*}

Our aim is to construct an isomorphism $\pi_1\Of \ \cong \pi_1\OfQ \wrm{n}\ZZZ$. 
Due to (2) of Theorem~\ref{th:fibration_DMX_Of} we have isomorphisms: 
\begin{align*}
\pi_1(\DTC, \SfC) & \ \cong \ \pi_1 \OfC, &
\pi_1(\DT, \Sf) & \ \cong \ \pi_1 \Of, &\
\pi_1(\DQ, \SfQ) & \ \cong \ \pi_1 \OfQ,
\end{align*}
and so we are reduced to finding an isomorphism
\begin{equation}\label{equ:required_iso}
\xi: \pi_1(\DQ, \SfQ) \wrm{n} \ZZZ \cong \pi_1(\DT, \Sf).
\end{equation}

Let $i:(\DTC, \SfC) \subset (\DT, \Sf)$ be the inclusion map.
It yields a morphism between the exact sequences of homotopy groups of these pairs, see Theorems~\ref{th:serre_fibr} and~\ref{th:fibration_DMX_Of}.
The non-trivial part of this morphism is contained in the following commutative diagram:
\begin{equation}\label{equ:morphism_of_sequences}
\begin{CD}
1 @>{}>> 1 @>>> \pi_1(\DTC,\SfC) @>{\dc}>> \pi_0\SfC @>{}>> 1 \\
&& @V{}VV @V{\jO}VV @VV{\jZ}V \\
1 @>{}>> \pi_1\DT @>{\qhom}>> \pi_1(\DT,\Sf) @>{\partial}>> \pi_0\Sf @>{}>> 1
\end{CD}
\end{equation}

In this section we describe kernel and images of all homomorphisms from~\eqref{equ:morphism_of_sequences}, see Theorem~\ref{th:two_epimorphisms} below.
For $n=1$ this diagram is studied in~\cite{MaksymenkoFeshchenko:MS:2014}.

\medskip 
For $\dif\in\Sf$ we will denote by $[\dif]$ its isotopy class in $\pi_0\Sf$.
If $\dif\in\SfC$, then its isotopy class in $\pi_0\SfC$ will be denoted by $\cclass{\dif}$.
Evidently,
\[
\jZ\bigl(\cclass{\dif}\bigr) = [\dif].
\]
Similarly, for a path $\omega:(I,\partial I, 0) \longrightarrow (\DT,\Sf, \idT)$ we will denote by $[\omega]$ its homotopy class in $\pi_1(\DT,\Sf)$.
If $\omega(I,\partial I, 0) \subset (\DTC,\SfC, \idT)$, then we denote by $\cclass{\omega}$ is homotopy class in $\pi_1(\DTC, \SfC)$.
Again
\[
\jO\bigl(\cclass{\omega}\bigr) = [\omega].
\]
Recall also that the boundary homomorphism $\dc:\pi_1(\DTC,\SfC) \longrightarrow \pi_0\SfC$ is defined as follows: if $\omega: (I,\partial I, 0) \to (\DTC,\SfC,\idT)$ is a continuous path, then
\[ \dc\bigl(\cclass{\omega}\bigr) = \cclass{\omega(1)} \in\pi_0\SfC.\]

\begin{theorem}\label{th:two_epimorphisms}
In the notation above there exist two epimorphisms
\begin{align*}
&\eval:\pi_1(\DT, \Sf)  \longrightarrow \ZZZ,
&
&\krot:\pi_0\Sf \longrightarrow \ZZZ_n,
\end{align*}
such that the following diagram is commutative:
\begin{equation}\label{equ:extended_comm_diagram}
\begin{gathered}
\xymatrix@R=0.6cm@C=1.2cm{
	& & & 1 \ar[d] & \\
	& & 1 \ar[d] & \Theta \ar[d] & \\
	& & \pi_1(\mathcal{D}^{\id}_{\mathcal{C}}, \mathcal{S}_{\mathcal{C}}) \ar[r]_-{\cong}^-{\partial_{\mathcal{C}}}  \ar[d]^-{i_1} & \pi_0\mathcal{S}_{\mathcal{C}}  \ar[d]^{i_0}& \\
	1 \ar[r] & \mathcal{L}\times \mathcal{M} \ar[d]^{\mathrm{pr}} \ar[r]^-{q} & \pi_1(\mathcal{D}^{\id},\mathcal{S}) \ar[d]^{\varphi} \ar[r]^-{\partial} & \pi_0 \mathcal{S} \ar[d]^{k} \ar[r] & 1\\
	1 \ar[r]  & \mathcal{L} \ar[r]^-{\cdot n} & \mathbb{Z} \ar[d] \ar[r]^-{\mathrm{mod}\, n} & \mathbb{Z}_n \ar[d]  \ar[r] & 1\\
	& & 1 & 1 &
}
\end{gathered}
\end{equation}
Here the arrow $\xrightarrow{~\cdot n~}$ means a unique monomorphism associating to the generator $\flowParallel\in\subgroupParallel$ the number $n$.
Moreover, the following statements hold true.
\begin{enumerate}
\item[\rm(a)]
$\qhom\bigl(\subgroupMeridian\bigr) = \jO\circ\dc^{-1}\bigl(\Slide\bigr)$; 
\item[\rm(b)]
all rows and columns in diagram~\eqref{equ:extended_comm_diagram} are exact;
\item[\rm(c)]
there exists a path $\gamma:(I,\partial I, 0) \to (\DT, \Sf, \idT)$ such that 
\begin{align*}
\eval[\gamma]&=1, & \gamma(1)^n &= \idT.
\end{align*}
\end{enumerate}
\end{theorem}

\subsection*{Proof of (a)}
Let $\flowMeridian:\Torus\times I\to\Torus$ be the loop in $\pi_1\DiffT$ generating a subgroup $\subgroupMeridian$ of $\pi_1\DiffT$, see~\eqref{equ:LM_flows}.
Let also $\slide =\slide_0\circ\cdots\circ\slide_{n-1}$ be the product of slides along all curves in $\CrvList$, see~\eqref{equ:all_slides_product}, $\slide^{-1}$ be its inverse, and $\cclass{\slide^{-1}} \in \Slide$ be the isotopy class of $\slide^{-1}$ in $\pi_0\SfC$.
Then $\cclass{\slide^{-1}}$ also freely generates $\Slide=\langle\cclass{\slide}\rangle$.
Therefore it suffices to prove that 
\[
\qhom(\flowMeridian) = \jO\circ\dc^{-1}\bigl(\,\cclass{\slide^{-1}}\,\bigr).
\] 
Notice that $\qhom(\flowMeridian)$ is represented by the isotopy $\{\flowMeridian_t\}_{t\in I}$.

Also recall that we can also regard $\flowMeridian$ as a flow on $\Torus$ defined by the same formula~\eqref{equ:LM_flows}. 
Since all orbits of $\flowMeridian$ have period $1$, $\flowMeridian_{\alpha} = \flowMeridian_{\alpha+k}$ for all $k\in\ZZZ$ and any function $\alpha$.

In particular, by Lemma~\ref{lm:slide_properties} $\slide^{-1} = \flowMeridian_{-\sigma} = \flowMeridian_{1-\sigma}$ for a $C^{\infty}$ function $\sigma:\Torus\to\RRR$ such that $\sigma=1$ on a small neighbourhood $\Uman$ of $\CrvList$ and $\sigma=0$ outside some larger neighbourhood $\Nman$.

Now let $\mathbf{G}_t = \flowMeridian_{t(1-\sigma)}$, $t\in I$, be an isotopy between $\mathbf{G}(0)=\idT$ and $\mathbf{G}(1)=\slide^{-1}$ fixed on some neighbourhood of $\CrvList$. 
Regard it as a path $\mathbf{G}: (I, \partial I, 0) \longrightarrow (\DTC, \SfC, \idT)$.
Then $\partial(\cclass{\mathbf{G}}) = \cclass{\mathbf{G}(1)} = \cclass{\slide^{-1}}$, and so
\[\dc^{-1} \cclass{\slide^{-1}} = \cclass{\mathbf{G}}.\]
As $\dc$ is an \textit{isomorphism}, $\dc^{-1}\cclass{\slide^{-1}}$ \textit{does not depend} on a particular choice of such an isotopy $\mathbf{G}$.
Furthermore, $\jO\circ \dc^{-1}\cclass{\slide^{-1}}$ is a homotopy class of $\mathbf{G}$ regarded as a map
\begin{align}\label{equ:G_as_a_map_of_triples}
\mathbf{G}:& (I, \partial I, 0) \longrightarrow (\DT, \Sf, \idT),
&
\mathbf{G}(t)& = \flowMeridian_{t(1-\sigma)}.
\end{align}
Therefore it remains to show that $[\mathbf{G}] = \qhom(\idT\times\flowMeridian) \in \pi_1(\DT,\Sf)$.
In fact the homotopy between $\{\mathbf{G}_t\}_{t\in I}$ and $\{\flowMeridian_t\}_{t\in I}$ in the space $C\bigl( (I, \partial I, 0), (\DT, \Sf, \idT)\bigr)$ can be defined as follows:
\begin{align*}
&\mathbf{H}: (I, \partial I, 0) \times I \longrightarrow (\DT, \Sf, \idT),
&
\mathbf{H}(t,s) &= \flowMeridian_{t(1-s\sigma)}.
\end{align*}
We leave the details for the reader, see~\cite{MaksymenkoFeshchenko:MS:2014}.
	
\subsection*{Proof of (b)}
The upper row of~\eqref{equ:extended_comm_diagram} coincides with~\eqref{equ:morphism_of_sequences} and exactness of the lower row is evident.
Therefore it remains to construct epimorphisms $\eval$ and $\krot$ and prove that the columns of the diagram~\eqref{equ:extended_comm_diagram} are exact as well.

\subsection*{(b1) Construction of $\krot:\pi_0\Sf \longrightarrow \ZZZ_n$.}
Let $\dif\in\Sf$.
Then $\dif(\CrvList)=\CrvList$.
Since the curves in $\CrvList$ are \myemph{cyclically} ordered, there exists $\krot(\dif)\in\ZZZ_n$ such that
\begin{equation}\label{equ:krot_def}
\dif(\curveMeridian_{i}) = \curveMeridian_{i+\krot(\dif)\,\mathrm{mod}\,n}, \qquad i=0,\ldots,n-1.
\end{equation}
Recall that all indices here are taken module $n$.
Evidently, $\krot(\dif)$ depends only on the isotopy class $[\dif]$ of $\dif$ in $\Sf$, and the correspondence $\dif\longmapsto\krot[\dif]$ is a homomorphism $\krot:\pi_0\Sf \to \ZZZ_n$.
Moreover, $\krot$ is an epimorphism, since by definition $\CrvList$ consists of all images of $\curveMeridian$ with respect to $\Sf$.

\subsection*{(b2) Construction of $\eval:\pi_1(\DT, \Sf) \longrightarrow \ZZZ$.}
Let $\ncover:\RRR\times S^1\longrightarrow \Torus\equiv S^1\times S^1$ be the covering map defined by $\ncover(x,y) = \bigl(\tfrac{x}{n}\,\mathrm{mod}\,1, y\bigr)$.
Since $\curveMeridian_i = \frac{i}{n} \times S^1$, we have that 
\begin{equation}\label{equ:ncover_prop}
 \ncover(\{i\}\times S^1)= \curveMeridian_{i\,\mathrm{mod}\,n}, \qquad i\in\ZZZ,
\end{equation}
and in particular, $\ncover^{-1}(\CrvList) = \ZZZ\times S^1$.

\medskip
Let $\omega:(I,\partial I, 0) \longrightarrow (\DT,\Sf, \idT)$ be a representative of some element of $\pi_1(\DT, \Sf)$.
Then $\omega$ can be regarded as an isotopy $\omega:\Torus\times I \to \Torus$ such that $\omega_0=\idT$ and $\omega_1\in\Sf$, that is $\omega_1(\CrvList)=\CrvList$.
Therefore $\omega$ lifts to a unique isotopy $\widetilde{\omega}:(\RRR\times S^1)\times I \to \RRR\times S^1$ such that $\widetilde{\omega}_0=\id_{\RRR\times S^1}$ and $\ncover\circ\widetilde{\omega}_t = \omega_t \circ\ncover$ for all $t\in I$.

In particular, since $\omega_1(\CrvList)=\CrvList$, we have from~\eqref{equ:ncover_prop} that $\widetilde{\omega}_1(\ZZZ\times S^1) = \ZZZ\times S^1$, whence there exists an integer number $\eval_{\omega}\in\ZZZ$ such that 
\begin{equation}\label{equ:tilde_omega_prop}
 \widetilde{\omega}_1(\{i\}\times S^1) = ( \{i+\eval_{\omega}\}\times S^1), \qquad i\in\ZZZ.
\end{equation}
It is easy to see that $\eval_{\omega}$ depends only on the homotopy class $[\omega]$ of $\omega$ in $\pi_1(\DT, \Sf)$ and the correspondence $[\omega]\longmapsto\eval_{\omega}$ is a homomorphism $\eval:\pi_1(\DT, \Sf)  \longrightarrow \ZZZ$.

\subsection*{(b3) Commutativity of diagram~\eqref{equ:extended_comm_diagram}.}
Due to~\eqref{equ:morphism_of_sequences} the upper square is commutative.

\medskip
{\bf Lower right square}.
We need to check that
\begin{equation}\label{equ:k_d__eta_mod_n}
\krot\circ\partial = \eval \, \mathrm{mod} \, n.
\end{equation}
In the notation of (b2), notice that $\partial[\omega] = [\omega_1] \in \pi_0\Sf$ by definition of boundary homomorphism.
Hence for $i=0,\ldots,n-1$,
\[
\omega_1(\curveMeridian_i) \stackrel{\eqref{equ:ncover_prop}}{=\!=}
\omega_1\circ\ncover(\{i\} \times S^1)  =
\ncover\circ\widetilde{\omega}_1(\{i\}\times S^1) \stackrel{\eqref{equ:tilde_omega_prop}}{=\!=}
\ncover(\{i+\eval[\omega]\}\times S^1) = \curveMeridian_{i+\eval[\omega]\,\mathrm{mod}\,n}. 
\]
Now~\eqref{equ:k_d__eta_mod_n} follows from~\eqref{equ:krot_def}.

\medskip
{\bf Lower left square}.
We should show that
\begin{equation}\label{equ:eta_q_L__n}
 \eval\circ\qhom([\flowParallel]) = n.
\end{equation}
Evidently, the path $q(\flowParallel): (I,\partial I, 0) \longrightarrow (\DT,\Sf, \idT)$ can be regarded as an isotopy
\[ \flowParallel:\Torus\times I\to\Torus,
 \qquad 
 \flowParallel(x,y,t) = (x+\,\mathrm{mod}\,n, y),
\]
for $(x,y)\in\Torus$, see~\eqref{equ:LM_flows}. 
Then $\flowParallel$ lifts to an isotopy $\widetilde{\flowParallel}:(\RRR\times S^1)\times I \to \RRR\times S^1$ given by $\widetilde{\flowParallel}(x,y,t) = (x+nt,y)$.
In particular, $\widetilde{\flowParallel}(\{i\}\times S^1) = \{i+n\}\times S^1$, whence by~\eqref{equ:tilde_omega_prop} $\eval\circ\qhom([\flowParallel]) = n$.

\subsection*{(b4) Exactness of right column}
We should prove that the following sequence 
\[
1 \longrightarrow \Slide \xrightarrow{~~\subset~~} \pi_0\SfC \xrightarrow{~~\jZ~~} \pi_0\Sf \xrightarrow{~~\krot~~} \ZZZ_n \longrightarrow 1
\]
is exact.
By definition $\Slide$ is a subgroup of $\pi_0\SfC$ and as noted above $\krot$ is an epimorphism.
Therefore we should check that $\Slide = \ker\jZ$ and $\jZ(\pi_0\SfC) = \ker\krot$.

\medskip 
{\bf Inclusion $\Slide \subset \ker\jZ$.}

Recall that each $\slide_i\in\StabIdf$, whence their product $\slide\in\StabIdf$ as well, and therefore $\jZ(\cclass{\slide}) = [\slide] = [\idT] \in \pi_0\Sf$.
This shows that $\Slide=\langle\cclass{\slide} \rangle \subset \ker(\jZ)$

\medskip 
{\bf Inverse inclusion $\Slide \supset \ker\jZ$.}

Notice that the kernel of $\jZ:\pi_0\SfC\to\pi_0\Sf$ consists of isotopy classes of diffeomorphisms in $\SfC$ isotopic to $\idT$ by $\func$-preserving isotopy, however such an isotopy should not necessarily be fixed on $\curveMeridian$.
In other words, if we denote 
\[
\kerjo \ := \ \Sidf \cap\DTC \ = \
\StabIdf \ \cap \ \Diff(\Torus,\curveMeridian),
\]
then
\begin{equation}\label{equ:kerj0__pi0_StabIdf_DiffIdTC}
\ker\jZ \ = \ \pi_0 \kerjo. 
\end{equation}
Evidently, $\SidfC = \Stabf \cap \DiffIdTC$ is the identity path component of $\kerjo$, whence
\[
\ker \jZ \ = \ \pi_0\kerjo \ = \ \kerjo /\SidfC. 
\]
Also notice that each slide $\slide_i\in\StabIdf$, whence their product $\slide\in\StabIdf$ as well.
On the other hand by Theorem~\ref{th:slides_gen_pi0G} $\slide \in \DTC$, whence
\[
\slide \in \Sidf \cap\DTC = \kerjo.
\]

\begin{lemma}\label{lm:ker_j0}
$\pi_0\kerjo = \langle\cclass{\slide} \rangle \cong \ZZZ$.
In other words, each $\dif\in\kerjo$ is isotopic in $\kerjo$ to $\slide^{b}$ for a unique $b\in\ZZZ$.
\end{lemma}
\begin{proof}
Let $\dif\in\kerjo$.
Since $\kerjo \ := \ \Sidf \cap \DTC  \ \subset \ \Sidf$, it follows from Lemma~\ref{lm:shift_functions} that there exists a unique smooth function $\alpha\in C^{\infty} (\Torus)$ such that $\dif = \flow_{\alpha}$.

Since $\dif$ is fixed on some neighbourhood $\Nman_i$ of $\curveMeridian_i$, that is $\dif(x) = \flow_{\alpha}(x) = \flow(x,\alpha(x)) = x$ for all $x\in \Nman_i$, it follows that $\alpha(x)$ must be an integer multiple of the period of $\curveMeridian_i$.
Hence $\alpha$ takes a constant integer value on $\Nman_i$.

We claim that this value is the same for all $i=0,\ldots,n-1$.
Indeed, let $\Cyl_i$ be a cylinder bounded by $\curveMeridian_i$ and $\curveMeridian_{i+1}$ is isotopic to $\id_{\Cyl_i}$ relatively to some neighbourhood of $\partial\Cyl_i$, and $\dtw_i$ be a Dehn twist supported in $\Int{\Cyl_i}$ and defined by~\eqref{equ:special_twists}.
By Lemma~\ref{lm:DehnTwist_gen_DQdQ} the isotopy class of its restriction $\dtw_i|_{\Cyl_i}$ generates the group $\pi_0\Diff(\Cyl_i,\partial\Cyl_i)$.
Then it is easy to see that $\dif|_{\Cyl_i}$ is isotopic in $\Diff(\Cyl_i,\partial\Cyl_i)$ to $\dtw^{b}$ if and only if $\alpha(\Cyl_{i+1})-\alpha(\Cyl_{i}) = b$.
By assumption $\dif|_{\Cyl_i}$ is isotopic to $\id_{\Cyl_i} = \dtw_{i}^{0}$ relatively to $\partial\Cyl_i$, whence $\alpha(\Cyl_{i+1})-\alpha(\Cyl_{i}) = 0$ for all $i$.

Thus $\alpha$ takes the same constant integer value on all of $\CrvList$, which of course depends on $\dif$.
Denote this value by $k$.
Then the isotopy between $\dif=\flow_{\alpha}$ and $\slide^{k} = \flow_{k\sigma}$ in $\SfC$ can be given by the formula: $\dif_t = \flow_{(1-t)\alpha  + tk\sigma}$, see Lemma~\ref{lm:isotopy_between_shifts}.

It remains to note that since $\func$ has critical points inside each $\Cyl_i$, $\slide^{k}$ is not isotopic to $\slide^{l}$ for $k\not=l$.
\end{proof}

\medskip
{\bf Inclusion $\image(\jZ) \subset \ker(\krot)$.}
Let $\dif\in\SfC$, so $\dif$ is fixed on $\CrvList$, and in particular, $\dif(\curveMeridian_i)=\curveMeridian_i$ for all $i$.
Then by~\eqref{equ:krot_def}, $\krot\circ\jZ\bigl(\cclass{\dif}\bigr) = 0$, i.e. $\mathrm{image}(\jZ) \subset\ker(\krot)$.

\medskip
{\bf Inverse inclusion $\image(\jZ) \supset \ker(\krot)$.}
Let $\dif\in\Sf$ be such that $\krot[\dif]=0$, that is $\dif(\curveMeridian_i)=\curveMeridian_i$ for all $i$.
Since $\dif$ is isotopic to $\idT$, it also preserves orientation of each $\curveMeridian_i$, therefore by Lemma~\ref{lm:shift_functions} we can assume that $\dif$ is fixed on some neighbourhood of $\CrvList$ and such a replacement does not change the isotopy class $[\dif] \in\pi_0\Sf$.
So we can assume that $\dif\in \DiffIdT \cap \DiffTC = \grp$, see~\eqref{equ:group_of_slides}.
Then by Corollary~\ref{cor:slides_gen_pi0G} we can write
\[
\dif = \slide_1^{a_1} \circ \cdots \circ \slide_{n-1}^{a_{n-1}} \circ \gdif,
\]
for some $a_i\in\ZZZ$ and $\gdif\in\DiffIdTC$.
But each $\slide_i\in\StabIdf$, whence $[\dif]=[\gdif]\in\pi_0\Sf$ and 
\[
\gdif \in \Stabf \cap \DiffIdTC \ \equiv \ \SfC.
\]
In other words, $[\dif] = [\gdif] = \jZ\bigl(\cclass{\gdif}\bigr)$.
Thus $\image(\jZ) \supset\ker(\krot)$ as well.

\subsection*{(b5) Exactness of middle column}
We need to check that the following short sequence 
\[
1 \longrightarrow \pi_1(\DTC,\SfC) \xrightarrow{~~\jO~~} \pi_1(\DT, \Sf) \xrightarrow{~~\eval~~} \ZZZ \longrightarrow 1
\]
is exact.
Since $\partial$, $\krot$ and $\mathrm{mod}\,n$ are surjective, it follows from~\eqref{equ:k_d__eta_mod_n} that \myemph{\bfseries $\eval$ is surjective} as well.
Therefore it remains to verify that $\jO$ is injective and $\image(\jO) = \ker(\eval)$.

\medskip
{\bf Inclusion $\image(\jO) \subset \ker(\eval)$.}
Again using notation of (b2) suppose that $\omega:(I,\partial I, 0) \longrightarrow (\DTC,\SfC, \idT)$ is a representative of some element of $\pi_1(\DTC, \SfC)$.
Thus $\omega$ can be regarded as an isotopy of $\Torus$ fixed on $\CrvList$.
Therefore its lifting $\widetilde{\omega}:(\RRR\times S^1)\times I \to \RRR\times S^1$ is fixed on $\ZZZ\times S^1$, whence $\widetilde{\omega}_1(\{i\}\times S^1) =\{i\}\times S^1$ for all $i\in\ZZZ$.
Therefore by~\eqref{equ:tilde_omega_prop}, $\eval\circ\jO\bigl(\cclass{\omega}\bigr) = \eval[\omega] = 0$, i.e. $\omega\in\ker(\eval)$.

\medskip
{\bf Inverse inclusion $\image(\jO) \supset \ker(\eval)$.}
Let $x\in\pi_1(\DT,\Sf)$ be such that $\eval(x)=0$, i.e. $x\in\ker(\eval)$.
Then 
\[
0 = \eval(x)\,\mathrm{mod}\,n = \krot\circ\partial(x).
\]
Hence $\partial(x) \in \ker(\krot) = \image(\jZ) = \jZ(\Slide)$.
In other words, $\partial(x) = \jZ(\slide^k)$ for some $k\in\ZZZ$, where for simplicity of notation we denote by $\slide$ its isotopy class $\cclass{\slide} \in\pi_0\SfC$.

Put $y = \jO \circ \dc^{-1}(\slide^k) \in \pi_1(\DT,\Sf)$.
Then
\[
\partial(y) = \partial \circ \jO \circ \dc^{-1}(\slide^k) = 
\jZ \circ \dc \circ \dc^{-1}(\slide^k) = \jZ(\slide^k) = \partial(x).
\]
Hence $xy^{-1} \in \ker(\partial) = \image(\qhom)$.
In other words,
\[
x = \qhom(\flowParallel)^a \cdot \qhom(\flowMeridian)^b \cdot y
\]
for some $a,b\in \ZZZ$.

We claim that $a=0$, whence $x = \qhom(\flowMeridian)^b \cdot y$.
Indeed, since $\eval\circ\qhom(\flowParallel) = n$, $\eval\circ\qhom(\flowMeridian) =0$, and $\eval(y) = \eval\circ \jO \circ \dc^{-1}(\slide^k) = 0$ we see that
\[
0 = \eval(x) = \eval\bigl(\qhom(\flowParallel)^a \cdot \qhom(\flowMeridian)^b \cdot y\bigr) = a n + 0 + 0,
\]
and so $a=0$.

Moreover, by (a) $\qhom(\flowMeridian)=\jO\circ\dc^{-1}(\slide^{-1})$, whence
\[
x \ = \ \qhom(\flowMeridian)^b \cdot y \ = \ \jO\circ\dc^{-1}(\slide^{-b}) \ \cdot \ \jO \circ \dc^{-1}(\slide^k) \ = \
\jO\circ\dc^{-1}(\slide^{k-b}) \ \in \ \image(\jO).
\]

\subsection*{Proof of (c)}
For $n=1$, we can take $\gamma$ to be the constant path into $\idT$.
Therefore assume that $n\geq2$.

Let $\flowParallel_t:\Torus\to\Torus$, $t\in I$, be the isotopy defined by~\eqref{equ:LM_flows} and generating $\subgroupParallel$, and $\lambda = \flowParallel_{1/n}$, thus
\[
\lambda(x,y) = (x+\tfrac{1}{n} \ \mathrm{mod} \ 1, \ y ).
\]
In fact we will use the following three properties of $\lambda$:
\begin{itemize}
\item
$\func\circ\lambda$ coincides with $\func$ on some neighbourhood $\Nman$ of $\CrvList$, see~\eqref{equ:relation_flowMer_flowHam};
\item
$\lambda^n = \idT$;
\item
$\lambda(\Cyl_i) = \Cyl_{i+1}$ for all $i=0,\ldots,n-1$.
\end{itemize}

Notice that by definition of cyclic index of $\func$, there exists $\dif\in\Sf$ such that $\dif(\Cyl_i) = \Cyl_{i+1}$ as well as $\lambda$.

\myemph{We can assume that $\dif = \lambda$ on some neighbourhood $\Nman$ of $\CrvList$}.
Indeed, since $\lambda$ and $\dif$ preserve orientation of $\Torus$, and $\func\circ\dif=\func$, it follows that $\dif\circ\lambda^{-1}$ leaves invariant all regular components of level sets of $\func$ belonging to $\Nman$.
Therefore $\dif$ is isotopic in $\Sf$ to a diffeomorphism $\dif_1\in\Sf$ such that $\dif_1\circ\lambda^{-1}$ is fixed on some neighbourhood $\Nman_1$ of $\CrvList$, whence $\dif_1 = \lambda$ near $\CrvList$.
Therefore we can replace $\dif$ with $\dif_1$ and $\Nman$ with $\Nman_1$.

\medskip 
\myemph{We can additionally assume that $\dif^n=\idT$}.
Indeed, we have that 
\[ \dif^{n-1}|_{\Nman} = \lambda^{n-1}|_{\Nman} = \lambda^{-1}|_{\Nman} = \dif^{-1}|_{\Nman}.\]
Define a diffeomorphism $\dif_1:\Torus\to\Torus$ by $\dif_1 = \dif$ on $\Mman\setminus\Cyl_{n-1}$, and $\dif_1 = \dif^{-1}$ on $\Cyl_{n-1}$.
Then $\dif_1$ is a well-defined diffeomorphism such that $\dif_1^n = \idT$ and $\func\circ\dif_1=\func$, i.e. $\dif_1\in\Stabf$.
Therefore we can again replace $\dif$ with $\dif_1$.

\medskip
\myemph{We claim that $\dif$ is isotopic to $\idT$}.
Indeed, since $\dif = \lambda$ on an open set, say on a neighbourhood of $\CrvList$, and $\gdif$ preserves orientation, we see that so does $\dif$.
But all non-trivial isotopy classes of diffeomorphisms of $\Torus$ have infinite orders, whence $\dif$ is isotopic to $\idT$. 

\medskip
Now let $\gamma_t:\Torus\to\Torus$, $t\in I$, be any isotopy between $\idT$ and $\dif$.
It can be regarded an element of $\pi_1(\DiffT,\Stabf)$.
Then $1 = \krot[\gamma]  = \eval[\gamma]\,\mathrm{mod}\,n$, so $\eval[\gamma] = an + 1$ for some $a\in\ZZZ$.
Replacing $\gamma$ with any representative of the class $[\gamma]\pmult[\flowParallel]^{-a}$ can assume that $\eval[\gamma]=1$.
Theorem~\ref{th:two_epimorphisms} is completed.

\section{$\func$-invariant free $\ZZZ_n$-action}\label{sect:free_zn_action}
The following theorem is a reformulation of (c) of Theorem~\ref{th:two_epimorphisms}.
It shows that there exists a free $\func$-invaraint $\ZZZ_n$-action on $\Torus$, and so $\func$ factors to a function of the same class $\FFF(\Torus)$ on the corresponding quotient $\Torus/\ZZZ_n$ being also a $\Torus$.

\begin{theorem}\label{th:f_invar_zn_action}
There exists an $n$-sheet covering map $p:\Torus\to\Torus$ and $\hfunc\in\FFF(\Torus)$ making commutative the following diagram:
\begin{equation}\label{equ:f_factors_to_n_covering}
\xymatrix{
\Torus \ar[rr]^-{p} \ar[dr]_-{\func}  & & \Torus \ar[dl]^-{\hfunc} \\
& \RRR 
}
\end{equation}
Moreover, the KR-graph of $\hfunc$ also has one cycle, however the cyclic index of $\hfunc$ is $1$.
\end{theorem}
\begin{proof}
Let $\gamma$ be the same as in (c) of Theorem~\ref{th:two_epimorphisms} and let $\gdif = \gamma(1) \in \Stabf$.
Then $\gdif^n=\idT$.
Notice also that $\gdif$ has no fixed points, since $\krot(\gdif)=\eval(\gamma)\,\mathrm{mod}\,n = 1$, i.e. $\gdif(\Cyl_i) = \Cyl_{i+1}$ for all $i$.
In other words, $\gdif$ yields a free $\func$-invariant action of $\ZZZ_n$ on $\Torus$ by orientation preserving diffeomorphisms.
Hence the corresponding factor map $p:\Torus\to\Torus/\ZZZ_n$ is an $n$-sheet covering of $\Torus$ and the factor space $\Torus/\ZZZ_n$ is diffeomorphic to $\Torus$.

Furthermore, since the action is $\func$-invariant, we obtain that $\func$ yields a smooth function $\hfunc:\Torus/\ZZZ_n=\Torus\to\RRR$, such that the diagram~\eqref{equ:f_factors_to_n_covering} becomes commutative.

It remains to note that since $p$ is a local diffeomorphism, the function $\hfunc$ has property (L) as well as $\func$.
Therefore $\hfunc\in\FFF(\Torus/\ZZZ_n)$.
The verification that KR-graph of $\hfunc$ has one cycle and that the cyclic index of $\hfunc$ is $1$ we leave for the reader.
\end{proof}

\section{Proof of Theorem~\ref{th:main:pi1Of}}\label{sect:proof:th:main:pi1Of}
We have to construct an isomorphism
\[
\xi: \pi_1(\DQ, \SfQ) \wrm{n} \ZZZ \cong \pi_1(\DT, \Sf).
\]
Let $\gamma:(I,\partial I, 0) \longrightarrow (\DT,\Sf, \idT)$ be a path defined in (c) of Theorem~\ref{th:two_epimorphisms}, and $\gdif = \gamma(1)\in\Sf$.
Then $\gdif(\Cyl_i) = \Cyl_{i+1}$ and $\gdif^n=\idT$.

Recall also that the group $\ZZZ$ acts on $\Maps(\ZZZ_n,\pi_1\OfQ)$ by formula~\eqref{equ:action_k_on_alpha}.

\begin{lemma}
There exists an isomorphism
\[
\eta: \Maps\bigl(\ZZZ_n,\pi_1(\DQ, \SfQ)\bigr) \longrightarrow \pi_1(\DTC,\SfC).
\]
Moreover, let $\alpha\in\Maps\bigl(\ZZZ_n,\pi_1(\DQ, \SfQ)\bigr)$, $k\in\ZZZ$, and $\alpha^k\in\Maps\bigl(\ZZZ_n,\pi_1(\DQ, \SfQ)\bigr)$ be the result of the action of $k$ on $\alpha$, see~\eqref{equ:action_k_on_alpha}.
Then
\begin{equation}\label{equ:g_act_on_paths}
\jO(\eta(\alpha^k)) = [\gamma^k] \pmult \jO(\eta(\alpha)) \pmult [\gamma^{-k}].
\end{equation}
\end{lemma}
\begin{proof}
Let $\alpha: \ZZZ_n\to \PQ$ be any map, and $\omega_i:(I,\partial I, 0)\to (\DQ, \SfQ, \idQ)$ be a representative of $\alpha(i)$ in $\pi_1(\DQ, \SfQ)$.
Then $\omega_i(t)$ is fixed near $\dQ$, whence we have a path $\omega:I\to\DTC$ given by  
\begin{equation}\label{equ:omega_for_alpha}
\omega(t)|_{\Cyl_i} = \gdif^{i} \circ \omega_i(t) \circ \gdif^{-i}|_{\Cyl_i}, \qquad i=0,\ldots,n-1.
\end{equation}
Notice that
\begin{align*}
\omega(0)|_{\Cyl_i} &= \gdif \circ \omega_i(0) \circ \gdif^{-i} = \id_{\Cyl_i},
&
\func\circ\omega(1)|_{\Cyl_i} &= \func \circ \gdif \circ \omega_i(1) \circ \gdif^{-i} = \func,
\end{align*}
whence $\omega(0)=\idT$ and $\omega(1)\in\Stabf \cap \DiffIdTC = \SfC$.
Therefore $\omega$ is a map of triples $\omega:(I,\partial I, 0)\to (\DTC, \SfC, \idT)$, and so it represents some element $\cclass{\omega}$ of $\pi_1(\DTC,\SfC)$.
It is easy to see that the class $\cclass{\omega}$ depends only on the classes of $[\omega_i]\in\PQ$.

Define the map $\eta: \Maps\bigl(\ZZZ_n,\pi_1(\DQ, \SfQ)\bigr) \longrightarrow \pi_1(\DTC,\SfC)$ by $\eta(\alpha) = \cclass{\omega}$.
A straightforward verification shows that $\eta$ is a group isomorphism.
We leave the details for the reader.

\medskip
Now let $k\in\ZZZ$.
Then by definition of the action $\alpha^k(i) = \alpha(i+k\,\mathrm{mod}\,n)$, $i=0,\ldots,n-1$.
In particular, if $\omega_i:(I,\partial I, 0)\to (\DQ, \SfQ, \idQ)$ is a representative of $\alpha(i)$ in $\PQ$, then $\omega_{i+k\,\mathrm{mod}\,n}$ is a representative of $\alpha^k(i)$.
Therefore the path $\omega':I\to\DTC$ defined by 
\[
\omega'(t)|_{\Cyl_{i}} = \gdif^{i} \circ \omega_{i+k\,\mathrm{mod}\,}(t) \circ \gdif^{-i}|_{\Cyl_i}, \qquad i=0,\ldots,n-1.
\]
corresponds to $\alpha^k$, that is $\eta(\alpha^k) = \cclass{\omega'}$.
Notice that 
\[  
\omega'(t)|_{\Cyl_{i}} = \gdif^{-k} \circ \gdif^{i+k} \circ \omega_{i+k\,\mathrm{mod}\,}(t) \circ \gdif^{-i-k} \circ \gdif^{k} |_{\Cyl_i} 
=\gdif^{-k} \circ \omega(t) \circ \gdif^{k} |_{\Cyl_i}.
\] 
Hence 
\[
\omega'(t) = \gdif^{-k} \circ \omega(t) \circ \gdif^{k} = 
\gamma_1^k \circ \omega_t \circ \gdif_1^{-k}.
\]
Notice that $\jO(\eta(\alpha))=[\omega]$ and $\jO(\eta(\alpha^k))=[\omega']$ are the homotopy classes of $\omega$ and $\omega'$ regarded as elements of $\pi_1(\DT,\Sf)$.
Then by~\eqref{equ:mult_pi1Of_conj}
\[
\jO(\eta(\alpha^k)) = [ \gamma_1^k \circ \omega_t \circ \gdif_1^{-k} ] =
[\gamma_t^k] \pmult [\omega_t] \pmult [\gamma_t^{-k}] =
[\gamma_t^k] \pmult \jO(\eta(\alpha)) \pmult [\gamma_t^{-k}].
\]
Lemma is proved.
\end{proof}

The following statements completes Theorem~\ref{th:main:pi1Of}.
\begin{lemma}
Define a map $\xi:\pi_1(\DQ, \SfQ) \wrm{n} \ZZZ \longrightarrow \pi_1(\DT,\Sf)$ by
\[
\xi(\alpha,k) = \jO(\eta(\alpha))\pmult[\gamma_t^k],
\]
for $\alpha\in\Maps(\ZZZ_n, \pi_1(\DQ, \SfQ))$ and $k\in\ZZZ$.
Then $\xi$ is a homomorphism making commutative the following diagram with exact rows, see~\eqref{equ:exact_seq_for_wreath_prod}:
\[
\begin{CD}
1 @>>> \Maps\bigl(\ZZZ_n, \pi_1(\DQ, \SfQ)\bigr) @>{\zeta}>> \pi_1(\DQ, \SfQ) \wrm{n} \ZZZ @>{p}>> \ZZZ @>>> 1 \\
&& @V{\eta}V{\cong}V @V{\xi}VV @| \\
1 @>>> \pi_1(\DTC, \SfC) @>{\jO}>> \pi_1(\DT,\Sf) @>{\eval}>> \ZZZ @>>> 1.
\end{CD}
\]
Hence, by five lemma, $\xi$ is an isomorphism.
\end{lemma}
\begin{proof}
We should check that $\xi$ is an isomorphism.
Suppose $\alpha,\beta\in\Maps\bigl(\ZZZ_n, \pi_1(\DQ, \SfQ)\bigr)$ and $k,l\in\ZZZ$.
Then in $\pi_1(\DQ, \SfQ) \wrm{n} \ZZZ$ we have that  
\[
(\alpha,k)\pmult(\beta,l) = (\alpha \beta^k, k+l)
\]
whence
\[
\xi(\alpha,k) = \jO(\eta(\alpha))\pmult[\gamma^{k}],
\qquad
\xi(\beta,k) = \jO(\eta(\beta))\pmult[\gamma_t^l].
\]
On the other hand,
\begin{align*}
\xi(\alpha \beta^k, k+l)
&= \jO(\eta(\alpha \beta^k))\pmult[\gamma_t^{k+l}] \\
&= \jO(\eta(\alpha))\pmult \jO(\eta(\beta^k)) \pmult[\gamma_t^{k+l}] \qquad \qquad \qquad \text{by~\eqref{equ:g_act_on_paths}} \\
&= \jO(\eta(\alpha))\pmult[\gamma_t^{k}]\pmult \jO(\eta(\beta)) \pmult[\gamma_t^{-k}]\pmult[\gamma_t^{k+l}] \\
&= \jO(\eta(\alpha))\pmult[\gamma_t^{k}]\pmult \jO(\eta(\beta))\pmult[\gamma_t^{l}] \\
&= \xi(\alpha,k) \pmult \xi(\beta,l),
\end{align*}
and so $\xi$ is a homomorphism.
Moreover,  
\[ \xi\circ \zeta(\alpha) = \xi(\alpha, 0) = \jO\circ\eta(\alpha),\]
\[\eval\circ\xi(\alpha,k) = \eval(\eta(\alpha)\pmult[\gamma^{k}]) = 
\eval\circ\eta(\alpha) + \eval([\gamma^{k}]) = 0 + k = k = p(\alpha,k).\]
Hence the above diagram is commutative, and by five lemma $\xi$ is an isomorphism.
\end{proof}


\begin{thebibliography}{10}

\bibitem{BolsinovFomenko:1997}
{A.~V. Bolsinov and A.~T. Fomenko}, \textit{Vvedenie v topologiyu integriruemykh
  gamiltonovykh sistem ({I}ntroduction to the topology of integrable
  hamiltonian systems)}, ``Nauka'', Moscow, 1997 (Russian). 

\bibitem{FomenkoFuks:HomotTopology:1989}
A.~T. Fomenko and D.~B. Fuks, \textit{Kurs gomotopicheskoi topologii}, ``Nauka'',
  Moscow, 1989, With an English summary. 

\bibitem{Hatcher:AlgTop:2002}
Allen Hatcher, \textit{Algebraic topology}, Cambridge University Press,
  Cambridge, 2002. 

\bibitem{Kronrod:UMN:1950}
A.~S. Kronrod, \textit{On functions of two variables}, Uspehi Matem. Nauk (N.S.)
  \textbf{5} (1950), no.~1(35), 24--134. 

\bibitem{Kudryavtseva:MatSb:1999}
E.~A. Kudryavtseva, \textit{Realization of smooth functions on surfaces as height
  functions}, Mat. Sb. \textbf{190} (1999), no.~3, 29--88. 

\bibitem{Kudryavtseva:MathNotes:2012}
E.~A. Kudryavtseva, \textit{The topology of spaces of {M}orse functions on surfaces}, Math.
  Notes \textbf{92} (2012), no.~1-2, 219--236, Translation of Mat. Zametki
  {{\bf{9}}2} (2012), no. 2, 241--261.

\bibitem{Kudryavtseva:MatSb:2013}
E.~A. Kudryavtseva, \textit{On the homotopy type of spaces of {M}orse functions on
  surfaces}, Mat. Sb. \textbf{204} (2013), no.~1, 79--118. 

\bibitem{Kulinich:MFAT:1998}
E.~V. Kulinich, \textit{On topologically equivalent {M}orse functions on
  surfaces}, Methods Funct. Anal. Topology \textbf{4} (1998), no.~1, 59--64.
  

\bibitem{Maksymenko:TA:2003}
Sergiy Maksymenko, \textit{Smooth shifts along trajectories of flows}, Topology
  Appl. \textbf{130} (2003), no.~2, 183--204. 

\bibitem{Maksymenko:AGAG:2006}
Sergiy Maksymenko, \textit{Homotopy types of stabilizers and orbits of {M}orse functions on
  surfaces}, Ann. Global Anal. Geom. \textbf{29} (2006), no.~3, 241--285.


\bibitem{Maksymenko:MFAT:2010}
Sergiy Maksymenko, \textit{Functions on surfaces and incompressible subsurfaces}, Methods
  Funct. Anal. Topology \textbf{16} (2010), no.~2, 167--182.

\bibitem{Maksymenko:ProcIM:ENG:2010}
Sergiy Maksymenko, \textit{Functions with isolated singularities on surfaces}, Geometry and
  topology of functions on manifolds. Pr. Inst. Mat. Nats. Akad. Nauk Ukr. Mat.
  Zastos. \textbf{7} (2010), no.~4, 7--66.

\bibitem{Maksymenko:UMZ:ENG:2012}
Sergiy Maksymenko, \textit{Homotopy types of right stabilizers and orbits of smooth
  functions functions on surfaces}, Ukrainian Math. Journal \textbf{64} (2012),
  no.~9, 1186--1203 (Russian).

\bibitem{Maksymenko:DefFuncI:2014}
Sergiy Maksymenko, \textit{Deformations of functions on surfaces by isotopic to the
  identity diffeomorphisms},  (2014), arXiv:math/1311.3347.

\bibitem{Maksymenko:orbfin:2014}
Sergiy Maksymenko, \textit{Finiteness of homotopy types of right orbits of morse functions
  on surfaces},  (2014), arXiv:math/1409.4319 (English).

\bibitem{Maksymenko:pi1Repr:2014}
Sergiy Maksymenko, \textit{Structure of fundamental groups of orbits of smooth functions on
  surfaces},  (2014), arXiv:math/1408.2612 (English).

\bibitem{MaksymenkoFeshchenko:UMZ:ENG:2014}
Sergiy Maksymenko and Bogdan Feshchenko, \textit{Homotopy properties of spaces of
  smooth functions on 2-torus}, Ukrainian Math. Journal \textbf{66} (2014),
  no.~9, 1205--1212 (Russian).

\bibitem{MaksymenkoFeshchenko:MS:2014}
Sergiy Maksymenko and Bogdan Feshchenko \textit{Orbits of smooth functions on 2-torus and their homotopy types},
   (2014), arXiv:math/1409.0502 (English).

\bibitem{MasumotoSaeki:KJM:2011}
Yasutaka Masumoto and Osamu Saeki, \textit{A smooth function on a manifold with
  given {R}eeb graph}, Kyushu J. Math. \textbf{65} (2011), no.~1, 75--84.
 

\bibitem{Reeb:ASI:1952}
Georges Reeb, \textit{Sur certaines propri\'et\'es topologiques des vari\'et\'es
  feuillet\'ees}, Actualit\'es Sci. Ind., no. 1183, Hermann \& Cie., Paris,
  1952, Publ. Inst. Math. Univ. Strasbourg 11, pp. 5--89, 155--156.
 

\bibitem{Sergeraert:ASENS:1972}
Francis Sergeraert, \textit{Un th\'eor\`eme de fonctions implicites sur certains
  espaces de {F}r\'echet et quelques applications}, Ann. Sci. \'Ecole Norm.
  Sup. (4) \textbf{5} (1972), 599--660. 

\bibitem{Sharko:UMZ:2003}
V.~V. Sharko, \textit{Smooth and topological equivalence of functions on
  surfaces}, Ukra\"\i n. Mat. Zh. \textbf{55} (2003), no.~5, 687--700.
 

\bibitem{Sharko:MFAT:2006}
V.~V. Sharko, \textit{About {K}ronrod-{R}eeb graph of a function on a manifold},
  Methods Funct. Anal. Topology \textbf{12} (2006), no.~4, 389--396.


\end{thebibliography}
\end{document}